\documentclass[12pt]{amsart}

\usepackage{amssymb,latexsym,amsthm, amsmath, comment, fullpage}
\usepackage[all]{xy}
\usepackage{graphicx, comment,  tikz-cd, libertine}

\ifx\pdfoutput\undefined
\usepackage{hyperref}
\else
\usepackage[pdftex,colorlinks=true,linkcolor=blue,urlcolor=blue]{hyperref}
\fi

\theoremstyle{plain}
\newtheorem{theorem}{Theorem}[section]

\newtheorem{corollary}{Corollary}[section]
\newtheorem{proposition}{Proposition}[section]

\newtheorem{condition}{Condition}

\newtheorem{lemma}{Lemma}[section]
\theoremstyle{definition}
\newtheorem{definition}{Definition}[section]

\newtheorem*{ack}{Acknowledgements}
\newtheorem*{Ass}{Standing Assumptions}

\newcommand{\vertiii}[1]{{\left\vert\kern-0.25ex\left\vert\kern-0.25ex\left\vert #1 
    \right\vert\kern-0.25ex\right\vert\kern-0.25ex\right\vert}}
   
\begin{document}

\title[The cohomological equation and cyclic cocycles for renormalizable minimal Cantor systems]
      {The cohomological equation and cyclic cocycles for renormalizable minimal Cantor systems}
      \author{Rodrigo Trevi\~no}
      \address{Department of Mathematics, The University of Maryland, College Park, USA}
      \email{rodrigo@trevino.cat}
      \date{\today}
      \begin{abstract}
        For typical properly ordered and minimal Bratteli diagrams $(B,\leq_r)$, it is shown that there are finitely many invariant distributions $\mathcal{D}_i$ which are the only obstructions to solving the cohomological equation $f = u-u\circ \phi$ for the corresponding adic transformation $\phi:X_B\rightarrow X_B$ and for $\alpha$-H\"older $f$ with $\alpha$ large enough. These invariant distributions are then used to define cyclic cocycles, a.k.a. traces $\tau:K_0(\mathcal{A}_\phi)\rightarrow \mathbb{R}$ for the crossed product algebra $\mathcal{A}_\phi$.
      \end{abstract}
      \maketitle

      {\centering \footnotesize  Dedicated to Giovanni Forni on the occassion of his 60th birthday. \par}

      \section{Introduction and statement of results}

	Consider a homeomorphism of a compact metric space $\phi:X\rightarrow X$. One way to measure the complexity of such a system is to look for the number of dynamical invariants of the system. At the most basic level, these can correspond to equivalence classes of functions which are invariant under the dynamics. However the classes may be defined, they should satisfy $[f] = [f\circ\phi ]$, or equivalently $[f-f\circ\phi]$ should be a trivial class. The problem then becomes characterizing obstructions to solutions $u$ of the cohomological equation $f = u-u\circ \phi$ for a given $f$. A function of the form $u-u\circ\phi$ is called a \textbf{coboundary}. A dual point of view is to look for functionals $\mathcal{D}$ which are invariant under the dynamics: $\mathcal{D} = \phi_*\mathcal{D}$. Such functionals would then be defined on classes of functions defined up to coboundaries and these types of functionals present obstructions for a function to be a coboundary. A natural question is to then ask how the number of invariants may depend on the regularity of the class of functions considered, which would also restrict the types of invariant functionals or distributions that would be considered as obstructions.

	These type of problems have a long history. The classical result of Gottschalk and Hedlund can be considered the first result in this direction for minimal homeomorphisms and continuous functions. For systems with hyperbolic behavior, the study of this equation generally goes under the name of Livsic theory, dating back to the work of Livsic \cite{livsic}. The work of Giovanni Forni \cite{forni:1997, forni:deviation} opened the door to these types of questions for parabolic systems -- systems which are minimal and exhibit polynomial rates of divergence for nearby orbits. Forni has shown that, unlike the case of systems with hyperbolic hyperbolic, parabolic systems can be identified by having more than one but finitely many invariant distributions for functions of sufficient regularity, and that the regularity of functions and the associated distributions play a subtle and important role. He pioneered the use of renormalization methods to uncover the existence of invariant distributions. Another important series of work in the parabolic setting is that of Marmi-Moussa-Yoccoz \cite{MMY:2005, MY:Holder} in solving the cohomological equation for interval exchange transformations, the work of Giulietti-Liverani \cite{GL:parabolic} using anisotropic Banach spaces for parabolic systems, and the work of Faure-Gou\"ezel-Lanneau \cite{FGL:anosov} using similar transfer operator methods for pseudo-Anosov maps\footnote{The literature on Livsic theory is vast and I will therefore make no attempt to highlight results for systems with hyperbolic behavior.}.

The focus on this paper is on minimal Cantor systems $\phi:X\rightarrow X$ which are renormalizable and thus in the parabolic realm. In this setting, due to the seminal work of Giordano, Putnam and Skau \cite{GPS}, dynamical cohomology for functions of the highest regularity contains a lot of information. More precisely, the cohomology $G_\mathbb{Z}(\phi):=C(X,\mathbb{Z})/(\mathrm{Id}-\phi^*)$, considered as an ordered abelian group with units, is an invariant of strong orbit equivalence, while the group $G_\mathbb{Z}(\phi)/\mathrm{Inf}_\mathbb{Z}(\phi)$ (the group $\mathrm{Inf}_\mathbb{Z}(\phi)$ is defined below), also considered as an ordered abelian group with unit, is an invariant of orbit equivalence for minimal $\mathbb{Z}$-actions on Cantor sets. The group $G_\mathbb{Z}(\phi)$ is also isomorphic to the first \v Cech cohomology of the mapping torus of $\phi$.

One may wonder what information is encoded in the dynamical cohomology for continuous functions. Define for $\mathbb{G}\in\{\mathbb{Z},\mathbb{R}\}$:
      $$G_\mathbb{G}(\phi):= C(X,\mathbb{G})/(\mathrm{Id}-\phi^*),$$
      the \textbf{group of $\mathbb{G}$-coinvariants of $\phi$} and 
      $$\mathrm{Inf}_\mathbb{G}(\phi) = \left\{[f]\in G_\mathbb{G}(\phi) : \mu(f) = 0\mbox{ for all $\phi$-invariant measures }\mu\right\},$$
      the \textbf{group of $\mathbb{G}$-infinitesimals of $\phi$}. In contrast with $G_\mathbb{Z}(\phi)$, the continuous dynamical cohomology is quite large. Indeed, it is known that if the set of $\phi$-invariant ergodic probability measures is finite -- as is the case in the systems of interest in this paper -- then the dimension of $\mathrm{Inf}_\mathbb{R}(\phi)$ is $|\mathbb{R}|$. Moreover, if the set of $\phi$-invariant ergodic probability measures is finite, then $\dim (G_\mathbb{R}(\phi)/\mathrm{Inf}_\mathbb{R}(\phi))$ is equal to the number of invariant ergodic probability measures (see \cite[\S 3]{ormes:cbdry} for more). Thus it is natural to wonder whether there is an intermediate amount of regularity -- between the extremes $C(X,\mathbb{R})$ and $C(X,\mathbb{Z})$ -- whose cohomology captures most of the invariants of the system\footnote{The cohomology $G_\mathbb{R}(\phi)$ can still be useful: by \cite{ormes:cbdry}, when used with $G_\mathbb{Z}(\phi)$, it can give information about the discrete part of the spectrum.}. This paper addresses this question.

By \cite{GPS}, every minimal Cantor system is defined by an adic/Vershik transformation on the path space $X_B$ of a properly ordered Bratteli diagram $(B,\leq_r)$ (see \S \ref{sec:brat} for all the relevant definitions). Let $\mathcal{O}_\beth$ be the space of all ordered Bratteli diagrams. There is a continuous map $\sigma:\mathcal{O}_\beth\rightarrow \mathcal{O}_\beth$ which takes an ordered Bratteli diagram $(B,\leq_r)$ and gives $(B',\leq_r') = \sigma(B,\leq_r)$ to be the Bratteli diagram obtained from truncating the first level of $B$ and keeping the orders from $\leq_r$. This map serves as the renormalization dynamics on the space of adic transformations.

In order to state the main result, some definitions need to be given. A measure $\mu$ on $\mathcal{O}_\beth$ is \textbf{minimal} if $\mu$-almost every ordered diagram $(B,\leq_r)$ is strongly minimal, meaning that its path space $X_B$ is a Cantor set. A measure $\mu$ on $\mathcal{O}_\beth$ is \textbf{proper} if $\mu$-almost every ordered diagram $(B,\leq_r)$ is properly ordered, that is, has a unique all minimum path $x^-$ and a unique all maximum path $x^+$. Note that by \cite{BKY:perfect} -- which shows that in many cases a randomly chosen order on a diagram is proper -- the requirement of a measure $\mu$ on $\mathcal{O}_\beth$ to be proper is not an unreasonable one. A technical condition on the measure will also be required. Let $V_k(B)$ be the vertex set of a Bratteli diagram $B$ at level $k$.

\begin{condition}
\label{cond:1}
For $\mu$-almost every $x = (B,\leq_r)$, for all $\varepsilon>0$ there is a $K_\varepsilon^x$ such that $|V_k(B)|\leq K_\varepsilon^x e^{\varepsilon k}$ for all $k\geq 0$.
\end{condition}
This condition is satisfied for measures supported on diagrams for which there is an upper bound on the number of vertices at every level.

	For a Cantor set $X$, let $H_\alpha(X)$ be the space of $\alpha$-H\"older continuous functions. Note first that this space is nontrivial for all $\alpha>0$, and also that there is an implicit choice of a metric. Finally, if $\phi:X_B\rightarrow X_B$ is an adic transformation defined on the path space of a Bratteli diagram $B$, denote by $\hat{X}_B$ the mapping torus of $\phi$. The main result of this paper is the following.
      \begin{theorem}
        \label{thm:main}
        Let $\mu$ be a minimal, proper, $\sigma$-invariant ergodic probability measure on $\mathcal{O}_\beth$ satisfying Condition \ref{cond:1}. Then there is a $d_\mu\in\mathbb{N}$ such that for $\mu$-almost every $x = (B,\leq_r)\in \mathcal{O}_\beth$, for any $\alpha>2$ there are $d_\mu$ obstructions to solving $f = u\circ\phi - u$ for $f \in H_\alpha(X_B)$. More specifically, for $\mu$-almost every $x = (B,\leq_r)$, $ d_\mu = \dim \check H^1(\hat{X}_B;\mathbb{R})$ and for any $\alpha>2$, there are $d_\mu$ $\phi$-invariant distributions $\mathcal{D}_1,\dots, \mathcal{D}_{d_\mu}\in H_\alpha(X_B)'$ such that $\mathcal{D}_1(f) = \cdots= \mathcal{D}_{d_\mu}(f) = 0$ if and only if for all $\varepsilon>0$ small enough there exists a constant $K = K(\alpha,\mu, x,\varepsilon)>0$ and $u\in H_{\alpha-2-\varepsilon}(X_B)$ such that $f = u\circ\phi - u$ and
        $$\|u\|_{\alpha-2-\varepsilon}\leq K\cdot \| f \|_{\alpha}.$$
      \end{theorem}
      Let me make some remarks about Theorem \ref{thm:main}. First, it should be pointed out first that the H\"older norms in the statement above are implicitly made using an ultrametric dependent on the Lyapunov exponent given by the renormalization dynamics, and therefore the loss of regularity in a way depends on the Lyapunov exponent.

      Second, the result shows that for functions which are regular enough, the $\alpha$-H\"older cohomology is finite dimensional, and so is the space of $\alpha$-H\"older infinitesimals and thus marking a transition between the results of \cite{ormes:cbdry} and those such as \cite{GPS} which consider the cohomology of $C(X_B,\mathbb{Z})$.
      
      Third, one may wonder what these results say in the case of interval exchange transformations (IETs). The answer is likely not much, and there are several reasons for this. First, Theorem \ref{thm:main} is about systems on Cantor sets and IETs are defined on intervals, so at a superficial level they cannot be compared. However, one may point out that minimal Cantor systems have IETs as factors (although usually of infinite type \cite{LT} ) and thus perhaps by lifting functions from the interval one may be able to analyize it at the Cantor level. Unfortunately, the required regularity at the Cantor level cannot be obtained by lifting a regular function on the interval by the semi-conjugacy map. But the most significant reason in my opinion is the following: the systems in Theorem \ref{thm:main} are those which are defined by diagrams which are properly ordered. The Bratteli diagrams given by Rauzy-Veech induction are far from properly ordered \cite[\S 6]{PT:brat}, and therefore one cannot hope to use that model with the results here. It may be the case that other models such as \cite{GjerdeJohansen} could be used in conjunction with the ideas of this paper, but this seems to require a lot more work without any guarantee that it would be a successful endeavour.

      Fourth, the systems considered here can be considered $S$-adic systems \cite{BD:sadic}. The reason they are presented through ordered Bratteli diagrams rather through the $S$-adic formalism simply reflects how I think of Cantor minimal systems. 
      
    The final remark is about the proof of Theorem \ref{thm:main}: it relies on the construction -- inspired by \cite{AP} -- of a foliated compact metric space $\Omega_x$ for the typical $x = (B,\leq_r)$. These spaces are inverse limits of expanding affine maps on graphs. These spaces are homeomorphic (in fact, close to bi-Lipschitz) to $\hat{X}_B$, and there is much to be gained from doing analysis on the spaces $\Omega_x$. Spaces of functions $\mathcal{S}^r_\alpha(\Omega_x)$, which were introduced in \cite{T:transversal}, are further developed here and used to study the leafwise tangential cohomology of $\Omega_x$. It is shown here that for functions of sufficient regularity, that is, for $f\in \mathcal{S}_\alpha^r(\Omega_x)$ with $r \geq 1$ and $\alpha$ large enough, the leafwise tangential cohomology in these spaces of functions is finite dimensional by proving an isomorphism to a de Rham-type of cohomology, and then comparing that to the \v Cech cohomology using a theorem of Sadun. Thus the strategy here to compute solutions to cohomological equations of parabolic flows is unlike other strategies because the setting here is not a smooth one ($\Omega_x$ are not smooth spaces) and therefore the function spaces $\mathcal{S}_\alpha^r$ were developed to do analysis in this setting. The cohomological information provided through the spaces $\mathcal{S}_\alpha^r$ in turn gives information about the dynamical cohomology of $\phi$ for functions of sufficient H\"older regularity. 

The invariant distributions $\mathcal{D}_i$ from the theorem above define cyclic cocycles, and this is the content of the second theorem. First, recall the crossed product $C^*$-algebra $\mathcal{A}_\phi:= C(X_B)\underset{\phi}{\times} \mathbb{Z}$ associated to the minimal Cantor system (see \cite{williams:book} or \cite{GKPT:notes}). The group $K_0(\mathcal{A}_\phi)$ is order isomorphic to $G_\mathbb{Z}(\phi)$ \cite{putnam:89}. If $\phi$ is uniquely ergodic (as are all systems considered in this paper), then there is a unique tracial state $\tau_1:\mathcal{A}_\phi\rightarrow \mathbb{R}$, that is, a linear functional $\tau$ satisfying $\tau(ab) = \tau(ba)$ with $\tau(1) = 1$. This defines a trace $\tau_{1*}:K_0(\mathcal{A}_\phi)\rightarrow \mathbb{R}$. The second result addresses the question of whether there are other functionals, perhaps defined in a dense subalgebra of $\mathcal{A}_\phi$, which satisfy $\tau(ab) = \tau(ba)$. These are called cyclic cocycles or traces.

The second result shows that for all $\alpha$ large enough there is a dense $*$-subalgebra $W^\infty_\alpha(X_B)\subset \mathcal{A}_\phi$ and traces $\tau_{i}:W^\infty_\alpha\rightarrow \mathbb{R}$ defined from the invariant distributions $\mathcal{D}_i\in H_\alpha(X_B)'$ from Theorem \ref{thm:main}. The second result shows that these traces can be seen as a full pairing to $K_0(\mathcal{A}_\phi)$.
      \begin{theorem}
        \label{thm:main2}
        Let $\mu$ be a minimal, proper, $\sigma$-invariant ergodic probability measure on $\mathcal{O}_\beth$. Then there is a $d_\mu\in\mathbb{N}$ such that for $\mu$-almost every $x = (B,\leq_r)\in \mathcal{O}_\beth$, there exist $d_\mu$ traces
        $$\tau_i:K_0(\mathcal{A}_\phi)\rightarrow \mathbb{R},$$
        with $i = 1,\dots, d_\mu$, with $\tau_1$ being the tracial state defined by the unique $\phi$-invariant ergodic probability measure on $X_B$. These traces are defined on a dense subalgebra $W^\infty_\alpha(X_B)\subset \mathcal{A}_\phi$ whose inclusion into $\mathcal{A}_\phi$ induces an isomorphism on $K_0$.
      \end{theorem}
      
      This paper is organized as follows. Section \ref{sec:brat} serves as a review for Bratteli diagrams and minimal Cantor systems. In \S \ref{sec:renorm}, the renormalization dynamics are introduced. Section \ref{sec:solenoid}, the inverse limit construction producing the solenoid $\Omega_x$ is introduced, along with their function spaces $\mathcal{S}^r_\alpha$, and properties of the homeomorphism between $\hat{X}_B$ and $\Omega_x$. In Section \ref{eqn:coh} the leafwise cohomology of $\Omega_x$ is computed for functions in $\mathcal{S}^r_\alpha$ and, from this, results and estimates are derived for the main theorem, Theorem \ref{thm:main}. Finally, in Section \ref{sec:traces}, the dense subalgebra $W^\infty_\alpha$ is defined, and it is proved that it is stable under holomorphic functional calculus, which implies Theorem \ref{thm:main2}.

      \begin{ack}
        This work was supported by the NSF award DMS-2143133 Career. I am grateful to Scott Schmieding, Ian Putnam, Jonathan Rosenberg and Giovanni Forni for very helpful discussions throughout the writing of this paper.
      \end{ack}
      \section{Bratteli diagrams and minimal Cantor systems}
      \label{sec:brat}
      \subsection{Bratteli diagrams}
      A \textbf{Bratteli diagram} is a graph $B = (V,E)$ partitioned as $V = \{V_k\}_{k\geq 0}$ and $E = \{E_k\}_{k\in\mathbb{N}}$ along with range and source maps $r,s:E\rightarrow V$ such that $r(E_k) = V_k$ for $k\in\mathbb{N}$ and $s(E_k) = V_{k-1}$, where it is assumed throughout this paper that the sets $V_k$ and $E_k$ are finite for all $k$. The same information can be defined through a family of matrices $\{A_k\}_{k\in\mathbb{N}}$, where $A_k$ is a $|V_{k}|\times |V_{k-1}|$ by letting $(A_k)_{vw}$ be the number of edges with source $v\in V_{k-1}$ and range $v\in V_k$.

      A (finite) \textbf{path} on a Bratteli diagram is a finite collection of edges $(e_m,\dots, e_{n-1})$, where $e_i\in E_i$ and $r(e_i) = s(e_{i+1})$ for all $i$. The range and source maps are extended to the set of all finite paths on a Bratteli diagram by setting $s((e_m,\dots, e_{n-1})) = s(e_m)$ and $r((e_m,\dots, e_{n-1})) = r(e_{n-1})$. Denote by $E_{m,n}$ to be the set of paths with source in $V_m$ and range in $V_n$. An infinite path is an infinite collection of edges $\bar{e} = (e_1,e_2,\dots)$ with $e_i\in E_i$ for all $i$ and $s(e_1) \in V_0$. A Bratteli diagram is \textbf{simple} or \textbf{minimal} if for each $k$ there exists an $\ell$ such that for any $v\in V_k$ and $w\in V_{k+\ell}$ there exists a path $p\in E_{k,k+\ell}$ with $s(p) = v$ and $r(p) = w$. Similarly, for $v\in V_k$, then $E_v\subset E_{0,k}$ is the set of paths $p$ such that $r(p) = v$. A diagram is \textbf{strongly minimal} if for each $k$ there exists an $\ell$ such that for any $v\in V_k$ and $w\in V_{k+\ell}$ there exist at least two distinct paths $p_1,p_2\in E_{k,k+\ell}$ with $s(p_i) = v$ and $r(p) = w$. All the diagrams which will considered in this paper will be strongly minimal.
      
      Let $X_B$ be the set of all infinite paths on $B$. It is topologized by cylinder sets of the form $C_p$, for some $p = (p_{m+1},\dots p_n)\in E_{m,n}$, defined as
      $$C_p = \{\bar{e} = (e_1,e_2,\dots)\in X_B: e_i = p_i\mbox{ for all }m<i\leq n\}.$$
      Another way to denote cylinder sets is, for $p\in X_B$, as
      $$C_k(p):= \left\{ q\in X_B: q_i = p_i \mbox{ for }i\leq k\right\}\mbox{ or }C_v:= \{p\in X_B: s(p_{k+1}) = v\}$$
      for a $v\in V_k$. Two paths $\bar{e}_1,\bar{e}_2$ are called \textbf{tail-equivalent} if there is a $k\geq 0$ such that the edges of the paths coincide starting at the $k^{th}$ level. This is an equivalence relation on $X_B$, and the tail-equivalence class of an element $p\in X_B$ will be denoted by $[p]$. A measure $\mu$ on $X_B$ is \textbf{invariant} under the tail-equivalence relation if for any two finite paths $p_1,p_2\in E_{0,k}$ with $r(p_1) = r(p_2)$ the measure satisfies $\mu(C_{p_1}) = \mu(C_{p_2})$. 
      
      For any $\lambda>1$ the function $d_\lambda:X_B^2\rightarrow \mathbb{R}$ defined by
      \begin{equation}
        \label{eqn:metricCantor}
        d_\lambda(x,y) := \lambda^{-k(x,y)+1},
      \end{equation}
      where $k(x,y)$ is the smallest index $i$ such that $x_i\neq y_i$, it is a metric for $X_B$, giving $X_B$ a basis of clopen sets $B_{\lambda^{-k}}(x) = C_{(x_1,\dots ,x_k)}$. Although it is common to work with $\lambda=2$, it will be more convenient to pick another value of $\lambda$ to make the computations and resulting statements easier to express. This does not change the topology of $X_B$. If $B$ is strongly minimal, then $X_B$ is a Cantor set and $d_\lambda$ is an ultrametric.
      
      \subsection{Minimal Cantor systems}
      An \textbf{ordered} Bratteli diagram $(B,\leq_r)$ consists of a Bratteli diagram $B$ along with a choice of order for the set $r^{-1}(v)$ for all $v\in V$. For $v\in V$, $e_m\in E$ is minimal in $r^{-1}(v)$ if it is the first element in this chosen order, whereas $e_M$ is maximal in $r^{-1}(v)$ if it is the last element in the chosen order. A path $\bar{e} = (e_1,e_2,\dots)\in X_B$ is a \textbf{minimal [resp. maximal] path} if $e_i$ is a minimal [resp. maximal] path in $r^{-1}(r(e_i))$ for all $i$. The set of minimal and maximal paths are denoted by $X^-_B,X^+_B\subset X_B$, respectively. If $|X^-_B| = |X^+_B| = 1$, the unique minimal and maximal paths will be denoted by $x^{min}$ and $x^{max}$, respectively.

      Let $(B,\leq_r)$ be an ordered Bratteli diagram. Then the order $\leq_r$ induces a lexicographic order of the set $E_v$ for any $v\in V$. The \textbf{successor} of $\bar{e}\in X_B$ determined by the order $\leq_r$, if it exists, is the path defined as follows. Let $q(\bar{e})\in \mathbb{N}$ be the smallest index such that $e_i$ is not maximal in $r^{-1}(r(e_i))$ (if it exists; otherwise the successor of $\bar{e}$ may not exist). Denote by $(e_1',e_2', \dots, e_{q(\bar{e})}')$ the path in $E_{r(e_{q(\bar{e})})}$ which follows $(e_1,\dots, e_{q(\bar{e})})$ in the lexicographic order in $E_{r(e_{q(\bar{e})})}$. Then the successor of $\bar{e}$ is the path $(e_1',e_2', \dots, e_{q(\bar{e})}', e_{q(\bar{e})+1},\dots)\in X_B$, which is well defined as long as $\bar{e}\not \in X^+_B$. The successor of $\bar{e}$ will be denoted by $\varphi(\bar{e})$.

      The successor map is a continuous function, and if $|X^-_B| = |X^+_B|=1$, then it extends to a bijection $\phi:X_B\rightarrow X_B$ by setting $\phi(x^-) = x^+$, where $x^\pm\in X^\pm_B$. This is a self-homeomorphism of $X_B$. Note that even if $|X^-_B| = |X^+_B|>1$ it may not be possible to extend the successor map of a homeomorphism of $X_B$. Ordered diagrams for which $|X^-_B| = |X^+_B|=1$ are called \textbf{properly ordered} diagrams. The map $\phi:X_B\rightarrow X_B$ is called the \textbf{Vershik transformation} or the \textbf{adic transformation}. If $B$ is minimal and properly ordered, then $\phi:X_B\rightarrow X_B$ is a minimal homeomorphism. If $B$ is strongly minimal, properly ordered, and $X_B$ is infinite, then $\phi:X_B\rightarrow X_B$ is a \textbf{minimal Cantor system}. Note that if $\mu$ is $\phi$-invariant, then this measure is also invariant under the tail-equivalence relation, and that this is independent of the order $\leq_r$ used to define the map $\phi$.

      \section{Renormalization}
      \label{sec:renorm}
      In this section the space of all Bratteli diagrams and the order bundle over it will be introduced, along with shift dynamics on this space which are play the role of ``renormalization'' dynamics for the adic transformation defined by the ordered diagram. 

      For a Bratteli diagram $B$, let $\mathcal{O}^r_B$ be the $r$-order space of $B$. That is, an element $\leq_r\in\mathcal{O}^r_B$ represents a choice of order of $r^{-1}(v)$ for all $v\in V$. Let $\beth$ be the space of Bratteli diagrams. This set is topologized by cylinder sets which prescribe, for some $k\in\mathbb{Z}$, both $|V_{k-1}|$ and $|V_k|$ and the edge set $E_k$ connecting these two sets. Since there are infinitely many choices for any particular site, the set $\beth$ is not locally compact.
      \begin{definition}
        The \textbf{order bundle} $\mathcal{O}_\beth$ over $\beth$ is the bundle which has as a fiber over $B\in\beth$ the $r$-order space $\mathcal{O}^r_B$.
      \end{definition}
      The order bundle is topologized by cylinder sets as follows. A basic cylinder set consists of a basic cylinder set of $\beth$, say by defining a transition matrix $A_k$ at a specific site $k\in\mathbb{Z}$, and then making a choice of ordering of the set $r^{-1}(v)$, for $v\in V_k$.
      
      The shift map $\sigma:\mathcal{O}_\beth\rightarrow \mathcal{O}_\beth$ is defined as follows. Let $(B,\leq_r)\in\mathcal{O}_\beth$ and denote by $\sigma(B,\leq_r)$ the ordered bi-infinite Bratteli diagram $(B',\leq_r') = (V',E',\leq_r')$ defined by $V'_k = V_{k+1}$ and $E'_k = E_{k+1}$, with $\leq_r'$ defined for $r^{-1}(v)$, $v\in V'_k$, by the order $\leq_r$ on $r^{-1}(\sigma^{-1}(v))$.
      \begin{definition}
        A measure $\mu$ on $\beth$ (or $\mathcal{O}_\beth$) is \textbf{minimal} if $\mu$-almost every $x = b$ (or $x = (B,\leq_r)$) is a strongly minimal Bratteli diagram. A measure $\mu$ on $\mathcal{O}_\beth$ is \textbf{proper} if $\mu$-almost every $x = (B,\leq_r)$ is a properly ordered  Bratteli diagram.
      \end{definition}
      Note that if $x = (B,\leq_r)$ is a strongly minimal or proper diagram (or both), then so is $\sigma(x)$. Thus, a $\sigma$-invariant ergodic probability measure on $\mathcal{O}_\beth$ assigns to strongly minimal or proper diagrams full or null measure.
      
      Let $\mathfrak{M}$ be the set of all integer-valued matrices and define $\sigma_*: \beth\rightarrow \mathfrak{M}$ to be the function assigning to $B\in \beth$ the integer matrix $A_1$. This fuction defines a linear cocycle
      $$\sigma_*:\mathbb{R}^{|V_0(B)|} \longrightarrow \mathbb{R}^{|V_1(B)|} = \mathbb{R}^{|V_0(\sigma(B))|}$$
      over the shift $\sigma$. Here the map will be refered to as the \textbf{renormalization cocycle}, and the product of matrices $A_k\cdots A_1$ will be denoted by $\mathcal{A}_{(k)}$.

      \begin{theorem}[Kesten-Furstenberg Theorem]
        Let $\mu$ be a minimal, $\sigma$-invariant ergodic probability measure on $\beth$. Then for $\mu$-almost every $B\in \beth$ there is a $\sigma$-invariant family of subspaces $Y_1^B\subset \mathbb{R}^{V_0(B)}$ and a Lyapunov exponent $\lambda_\mu>0$ such that
        \begin{enumerate}
        \item the family $\{Y_1^B\}_B$ is $\sigma_*$-invariant: $ \sigma_*Y_1^B = Y_1^{\sigma(B)}$, and
        \item for any vector $v\in \mathbb{R}^{|V_0|}$ with positive entries,
          $$\lim_{n\rightarrow \infty}\frac{\log\left\|\mathcal{A}_{(n)} v\right\| }{n} = \lambda_\mu.$$
        \end{enumerate}
      \end{theorem}

      There is a relationship between the adic transformation defined by a propertly ordered, minimal Bratteli diagram $(B,\leq_r)$ and that of $\sigma((B,\leq_r))$. Indeed, the Vershik transformation defined by $\sigma((B,\leq_r))$ can be seen as the transformation induced by returns to the set defined as the union of $e\in E_1$ such that $e$ is the minimal edge in $r^{-1}(v)$ for some $v\in V_1$.
      \section{Inverse limits}
      \label{sec:solenoid}
      In this section, compact metric spaces will be constructed to study the special flow over $s$-adic transformations. These spaces are known as one-dimensional Wieler solenoids. The construction is inspired by the construction of Anderson-Putnam \cite{AP} to study self-similar tilings. For the rest of this section the following assumptions will be made, so it is good to declare this now instead of repeating it over and over.
      \begin{Ass}
        The measure $\mu$ is a proper, minimal $\sigma$-invariant ergodic probability measure on the space of $r$-ordered Bratteli diagrams $\mathcal{O}_\beth$ and $x=(B,\leq_r)\in \mathcal{O}_\beth$ is a Kesten-Furstenberg-regular point. The measure $\mu$ also satisfies Condition \ref{cond:1}.
      \end{Ass}
	Note forst that the last assumption on $|V_k|$ is satisfied if there is an upper bound on the number of vertices at each level for $\mu$-almost every $(\mathcal{B},\leq_r)$. Moreover, for any such measure, $\mu$-almost every $x = (B,\leq_r)$ has a unique probability measure $\bar\nu$ which is invariant under the tail-equivalence relation in $X_B$. Let $l \in Y_1^{B}\subset \mathbb{R}^{V_0(B)}$ be the unique positive vector normalized so that
      \begin{equation}
        \label{eqn:roof}
        \sum_{v\in V_0(B)} l_v \bar\nu(C_v) = 1,
      \end{equation}
      where $\bar\nu$ is the unique tail-invariant probability measure on $X_B$.
      
      Denote by $x^\pm$ the unique min/max paths on $X_{B}$ under $\leq_r$. For each integer $k\geq 0$, and $v\in V_k$, let $e_v$ be a copy of the oriented CW complex $\left[0,\lambda_\mu^{-k}\left(\mathcal{A}_{(k)} l\right)_v\right]$ with boundaries $\partial^\pm e_v$, and define
      $$\Gamma_x^{k} = \left(\bigsqcup_{v\in V_{k}}e_v \right) /\sim$$
      where $\sim$ is the relation defined by $\partial^+ e_{v_1} \sim \partial^- e_{v_2}$ if and only if either one of the following cases holds:
      \begin{enumerate}
      \item there is a $\ell\in\mathbb{N}$, $w\in V_{k+\ell}$ and finite paths $\bar{e}_1,\bar{e}_2$ with $s(\bar{e}_i) = v_i$ and $r(\bar{e}_i) = w$, $i=1,2$, such that $ \bar{e}_2$ is the successor or $\bar{e}_1$ in the $\leq_r$ order of the finite set of paths $\bar{e}$ with $s(\bar{e})\in V_k$ and $r(\bar{e}) = w$, or
      \item $x^+_{k+1}\in s^{-1}(v_1)$ and $x^-_{k+1}\in s^{-1}(v_2)$.
      \end{enumerate}

      Each graph will be denoted by $\Gamma^k_x = (\mathbb{V}^k_x,\mathbb{E}^k_x)$, where $ \mathbb{V}^k_x$ is the set of vertices and $\mathbb{E}_x^k$ is the set of edges. As such, $|\mathbb{E}^k_x| = |V_k|$.
      
      
      Define now, for every $k\in\mathbb{N}$, a map $\gamma_k^x:\Gamma^k_x\rightarrow \Gamma^{k-1}_x $ as follows. For each $v\in V_k$ it needs to be defined what the image of $e_v\in \mathbb{E}^k_x$ is in $\Gamma_x^{k-1}$. So let $v\in V_k$ and denote the ordered edges as $r^{-1}(v) = \{e_1^v<\cdots < e_{|r^{-1}(v)|}^v\}$. Then $\gamma_k^x(e_v)$ is defined to be the image of first stretching $e_v$ by $\lambda_\mu$ and then mapping it to the path defined by the concatenation $\gamma_k^x(e_v) = e_{s(e_1^v)}e_{s(e_2^v)}\cdots e_{s(e^v_{|r^{-1}(v)|})}$ on $\Gamma_x^{k-1}$. To show that this is well defined, first, it should be noted by the definition of $\sim$ on $\Gamma_x^{k-1}$, there is a path defined by the concatenation of edges in the order of the definition of $\gamma_k^x(e_v)$. Moreover, the length of this path is
      \begin{equation}
        \begin{split}
          \sum_{i = 1}^{|r^{-1}(v)|} |e_{s(e_i^v)}| &= \sum_{e\in r^{-1}(v)} \lambda^{-(k-1)}_\mu \left(\mathcal{A}_{(k-1)} l \right)_{s(e)} =  \lambda_\mu^{-(k-1)}\sum_{e\in r^{-1}(v)}  \left(\mathcal{A}_{(k-1)} l \right)_{s(e)} \\
          &= \lambda^{-(k-1)}_\mu \left(A_k\mathcal{A}^\mathcal{B}_{(k-1)}l\right)_v = \lambda^{-(k-1)}_\mu \left(\mathcal{A}_{(k)}l\right)_v \\
          &= \lambda_\mu \cdot |e_v|
        \end{split}
      \end{equation}
      confirming that the image of the map for the edge $e_v$ is well defined, and thus the map $\gamma_k^x$ is well defined.

      The \textbf{Wieler solenoid} defined by $x = (B,\leq_r)$ is the inverse limit
      $$\Omega_x := \lim_{\leftarrow} \left(\Gamma_x^k,\gamma_k^x \right) = \left\{(z_0,z_1,\dots)\in \prod_{k\geq 0}\Gamma_x^{k}:z_{i-1} = \gamma_i^x(z_i) \mbox{ for all }i.\right\}.$$
      It comes equipped with projection maps $\pi_k^x:\Omega_x\rightarrow \Gamma_x^k$ which, by construction, satisfy the equation $\pi_{k-1}^x(z) = \gamma_k^x\circ\pi_k^x(z)$. For these maps to be continuous, the preimages $(\pi^x_k)^{-1}(I)$ of open sets $I\subset \Gamma^k_x $ are open in $\Omega_x$. If $\{I_{m,n}\}_m$ is a (countable) basis of open sets of $\Gamma^n_x$, then $\{(\pi^x_n)^{-1}(I_{m,n})\}_{m,n}$ is a basis of open sets of $\Omega_x$, and $\mathcal{A}$ is the smallest $\sigma$-algebra generated by these open sets. Note that this system gives an increasing sequence of $\sigma$-algebras
      $$\mathcal{A}_0\subset \mathcal{A}_1\subset \cdots \subset \mathcal{A}_k\subset \cdots \subset \mathcal{A},$$
      where $\mathcal{A}_k$ is the $\sigma$-algebra generated by the sets $\{(\pi^x_k)^{-1}(I_{m,k})\}_{m}$.
      
      Let $\mu_k$ be the Lebesgue measure on $\Gamma^k_x$. This is well-defined since lengths were specified for each edge of the graph $\Gamma^k_x$. There is a natural Borel probability measure $\bar\mu$ on $\Omega_x$ which is compatible with the collection of measures $\{\mu_k\}$ and projection maps $\pi_k^x$. Indeed, since the Borel $\sigma$-algebra $\mathcal{A}$ was generated by lifts of open sets in $\Gamma^k_x$ for all $k$, then it suffices to assign a measure to each lift of each element of a basis $\{I_{m,k}\}$. Thus the measure $\bar\mu$ is defined by the condition that $\pi_{k*}^x\bar\mu = \mu_k$ for all $k$.

      There is a natural $\mathbb{R}$-action on $\Omega_x$ defined as follows. Supposing for a moment that $z = (z_0,z_1,\dots )\in\Omega_x$ has the property that $[\pi^x_0(z_0)]$ is in the interior of some edge $e_v\in \mathbb{E}^0_x$, then for all $t\in\mathbb{R}$ small enough in absolute value, the translation of $z_0$ by $t$, $z_0+t\in e_v$, stays in the interior of the same edge of $\Gamma_x^0$ as $z_0$. Thus, the only way that the translation of $z_0$ by $t$ can be extended to all coordinates of $z$ is by defining
      \begin{equation}
        \label{eqn:translation}
        \varphi_t(z) = \left(z_0+t, z_1+\lambda_\mu^{-1}t, z_2 + \lambda_\mu^{-2}t,\dots \right),
      \end{equation}
      that is, by having $\pi_k^x(\varphi_t(z)) = \pi_k^x(z) + \lambda_\mu^{-k}t$ for all $k$.

      To show that this is well-defined, for $k\geq 0$, let $d_k(z) = \mathrm{dist}\left(\pi_k^x(z),\partial e_k(z)\right)$, where $e_k(z)\in\mathbb{E}^k_x$ denotes the edge in $\Gamma_k^x$ where $[\pi_k^x(z)]$ lies. That is, it is the distance from $z_k$ to the boundary of the edge wherein $z_k$ resides. Suppose that $[\pi_k^x(z)]$ is in the interior of an edge but for some $|t|<d_0(z)$, keeping $|s|<|t|$, $\lim_{s\rightarrow t}\pi_k^x(z)+ \lambda_\mu^{-k}s$ is a vertex of $\Gamma_k^x$. Then since $\gamma_n^x$ is a cellular map for all $n$,
      $$\lim_{s\rightarrow t} z_0+t = \lim_{s\rightarrow t}\gamma_1\circ\cdots \circ \gamma_k(z_k + \lambda^{-k}_\mu s)$$
      is a vertex of $\Gamma_0^x$, which contradicts that $d_0(z)>|t|$. Thus the translation of $z$ as defined in (\ref{eqn:translation}) is well defined at least when $d_0(z)>0$ for $|t|$ small enough.

      It remains to show how to define $\varphi_t(z)$ whenever $\pi_0^x(z)$ is a vertex of $\Gamma_x^0$. Suppose for the moment that $z_0 = \partial^+e_v$ for some $e_v\in\mathbb{E}^0_x$ and that there is some $k\in\mathbb{N}$ such that $\pi_k^x(z)$ is not a vertex of $\Gamma_x^k$. This means that there is a $k\in\mathbb{N}$ and $v\in V_k$ such that $\pi_k^x(z)$ is in the interior of an edge $e_v\in \mathbb{E}^k_x$. As such, $\pi_k^x(z)+\lambda^{-k}_\mu t$ is defined for all $|t|$ small enough, and this can be carried through the maps $\gamma_i^x$ to define $\pi_i^x(\varphi_t(z)) = \gamma_i\circ\cdots \circ \gamma_k(\pi_k^x(z)+\lambda^{-k}_\mu t)$ for all $0\leq i < k$. The same argument allows one to define $\varphi_t(z)$ whenever $z_0 = \partial^-e_v$ for some $e_v\in\mathbb{E}^k_x$ and $v\in V_k$ such that $\pi_k^x(z)$ is in the interior of an edge for $|t|$ small enough.

      Note that in the case that $z_0 = [\partial^+e_v]$ for some $e_v\in\mathbb{E}^0_x$ but that there is some $k\in\mathbb{N}$ such that $\pi_k^x(z)$ is not a vertex of $\Gamma_x^k$, the point $\pi_k^x(z)$ corresponds to a path $p\in E_v$, where $v\in V_k$ is the vertex corresponding to the edge $e_v\in \mathbb{E}^k_x$ containing $\pi_k^x(z)$. By virtue of being in the interior of $e_v$, the path $p\in E_v$ is neither maximal nor minimal.

      So now consider the case where $z_0 = [\partial^+e_v]$ for some $e_v\in\mathbb{E}^0_x$ and that there is no $k\in\mathbb{N}$ such that $\pi_k^x(z)$ is in the interior of an edge of $\Gamma_x^k$. This is only possible if $z$ comes from either the all min path $x^-$ or the all max path $x^+$. In that case, $\varphi_t(z)$ for small $t>0$ is given by
      $$\varphi_t(z) = \left(z'_0+t, z'_1+\lambda_\mu^{-1}t, z'_2 + \lambda_\mu^{-2}t,\dots \right),$$
      where $z'_k = [\partial^-e_v]$ for the unique $v\in V_k$ such that $s(x^-_k) = v$. $\varphi_t(z)$ for $t<0$ small enough is similarly defined by the uniqueness of the all max and all min paths $x^\pm$.
      
      Since $\Omega_x$ comes with an $\mathbb{R}$ action, for all $z$ and $s\in\mathbb{R}$ small enough, $\varphi_s(z)\rightarrow z$ as $s\rightarrow 0$. As such, a neighborhood of $z$ includes an interval isometric to $(-\varepsilon, \varepsilon)$ for $\varepsilon>0$ small enough. Moreover, since $\Omega_x$ inherits the (infinite) product topology, another form of convergence $z^k\rightarrow z$ takes place if $z^k_i = z_i$ for all $i\leq k$. That is, if $z^k$ and $z$ agree on many coordinates, then they are close. For $z\in\Omega_x$ and $k\in\mathbb{N}_0$, define the transverse cylinder sets as
      \begin{equation}
        \label{eqn:cylinder}
        C^\perp_k(z):= \left\{ z'\in\Omega_k: z'_i = z_i\mbox{ for all } i\leq k \right\} = (\pi^x_{k})^{-1}(z_k).
      \end{equation}
      Since $B$ is strongly minimal, given any $k\in \mathbb{N}$, there is a sequence $k_i\rightarrow\infty$ with $k_0 = k$ such that for any $v\in V_{k_i}$ and $w\in V_{k_{i+1}}$ there are at least two paths from $v$ to $w$. From this it follows that $C^\perp_k(z)$ is a Cantor set for any $k$ and $z$, and thus that the local product structure of $\Omega_x$ is that of a product of an interval and a Cantor set. Indeed, 

      The natural Borel probability measure $\bar\mu$ on $\Omega_x$ is invariant under the natural $\mathbb{R}$ action. This can be seen as follows. Let $z\in\Omega_x$ with $z_k = \pi^x_k(z)$ not a vertex in $\Gamma^k_x$, and denote by $\varepsilon_{z_k}$ the distance from $z_k$ to the closest vertex. Then for all $|s|<\varepsilon_{z_k}$ the translation of $z_k$ by $s$ on the edge wherein it resides remains on the same edge and is not a vertex. Let $\varepsilon<\varepsilon_{z_k}$. Then the $\varphi$-invariance of the measure $\mu$ follows from the observations that,
      \begin{equation}
        \label{eqn:measPres}
        \begin{split}
        \mu((\pi^x_k)^{-1}(B_\varepsilon(z_k))) &= \mu_k(B_\varepsilon(z_k)) = \mu_k\left(\varphi_{\lambda_\mu^{-k}s}(B_\varepsilon(z_k))\right) \\
        &= \mu\left( (\pi^x_k)^{-1}(\varphi_{\lambda^{-k}_{\mu}s}(B_\varepsilon(z_k))) \right) = \mu\left( \varphi_s((\pi^x_k)^{-1}(B_\varepsilon(z_k)))\right)
        \end{split}
        \end{equation}
      for $|s|<\varepsilon_{z_k}-\varepsilon$. The first equality holds by definition, the second by the invariance of Lebesgue measure on $\Gamma^k_x$ by translations, the third by definition, and the fourth by the definition of the sets being measured as long as $|s|<\varepsilon_{z_k}-\varepsilon$. By the local product structure of $\Omega_x$, the measure $\mu$ is locally of the form $\mathrm{Leb}\times \nu$, where $\nu$ is a transverse measure in the sense of \cite{BM:UE}.
      
      The solenoid $\Omega_x$ can be given a metric which generates the topology mentioned above. For $\lambda>1$, let
      \begin{equation}
        \label{eqn:fullMetric}
        \bar{d}_\lambda(y,z) = \sum_{k\geq 0}\lambda^{-k}\frac{d_k(y_k,z_k)}{\mathrm{diam}(\Gamma_x^k)},
      \end{equation}
      where $\mathrm{diam}(\Gamma_x^k)$ is the diameter of the graph $\Gamma_x^k$ and $d_k(y_k,z_k)$ is the distance between $y_k,z_k\in \Gamma_x^k$, which is well-defined by construction. Note that this metric assigns a diameter to all transverse cylinder sets of the form (\ref{eqn:cylinder}). In particular, if $y_i = z_i$ for all $i\leq k$ and $y_{k+1}\neq z_{k+1}$, then $d(y,z)\leq \frac{\lambda^{-k}}{\lambda-1}$. In other words,
      \begin{equation}
        \label{eqn:diameter}
        \mathrm{diam}_\lambda(C^\perp_k(z))\leq \frac{\lambda^{-k}}{\lambda-1}.
      \end{equation}
      Define the subset $\mho_x\subset \Omega_x$ by
      $$\mho_x:= \left\{z\in\Omega_x: \gamma_1^x\circ\cdots\circ \gamma_k^x(z_k) = [\partial^\pm e_v]\mbox{ for some }v\in V_0 \mbox{ and all }k\in\mathbb{N} \right\}  = \pi^{-1}_0\left(\mathbb{V}^0_x\right)$$
      and let $\hat{X}_{B}$ be the suspension of $\phi:X_{B}\rightarrow X_{B}$ with roof function $l$ defined in (\ref{eqn:roof}), that is, $l(p) = l_v$ if $s(p) = v$.
      \begin{proposition}
        \label{prop:almostBiLip}
        Let $\mu$ be a proper, minimal ergodic probability measure on $\mathcal{O}_\beth$. For $\mu$-almost every $x = (B,\leq_r)$, the natural $\mathbb{R}$ action on $\Omega_x$ is conjugate to the special flow over $\varphi:X_{B}\rightarrow X_{B}$ with roof function $l$. The conjugacy is given by an map $\varpi_x:\hat{X}_{B}\rightarrow \Omega_x$ with the property that $\varpi_x(X_{B}) = \mho_x$ such that for any $\varepsilon\in(0,1)$ there is a $K = K(\varepsilon,\mu, B)$ such that
        $$K d_{\lambda_\mu}(p,q)^{1+\varepsilon}\leq \bar{d}_{\lambda_\mu}(\varpi(p),\varpi(q)) \leq \frac{d_{\lambda_\mu}(p,q)}{\lambda_\mu-1}$$
        for any $p,q\in X_B$.
      \end{proposition}
      \begin{proof}
        First, note that the maps $\{\gamma_{k}^x\}$ can be composed yielding the family of maps
        $$\gamma_{m,n}^x:=\gamma^x_{m+1}\circ \cdots \circ \gamma^x_n:\Gamma^k_n\rightarrow \Gamma^k_{m}$$
        for $m<n$. Now a function $\varpi_x:X_{B}\rightarrow \mho_x\subset \Omega_x$ will first be defined and then shown to satisfy the conjugacy property desired.
        
        By construction, the map $\gamma_{m,n}^x$ takes an edge $e'\in \mathbb{E}^n_x$, stretches it by $\lambda_\mu^{n-m}$, and wraps it around edges of $\Gamma^m_x$ according to the order of the set of paths starting at $V_m$ and terminating in the vertex in $V_n$ corresponding to the edge $e'$ in $\Gamma^n_x$. As such, there is an order-preserving bijection between the set
        $$(\gamma_{0,k}^x)^{-1}\left( \mathbb{V}^0_x\right)\cap [e_v\setminus \partial^+e_v]\subset \Gamma_x^k$$
        and the set of paths $E_v$, where the order on $[e_v\setminus \partial^+e_v]$ is inherited from the orientation of the edge. That is, for $p = p_1p_2\cdots\in X_B$ and each $k\in\mathbb{N}$, the order $\leq_r$ gives the finite set of paths $E_{r(p_{k})}$ which start at $V_0$ and end at $r(p_{k})$, and these paths can be matched, in order through bijection with the set $(\gamma_{0,k}^x)^{-1}\left(\mathbb{V}^0_x\right)\cap [e_{r(p_k)}\setminus \partial^+e_{r(p_k)}]$. 
        
        The function $\varpi:X_B\rightarrow \mho_x$ can now be defined. For $p = p_1p_2\cdots\in X_B$ and $k\in\mathbb{N}$, let $\varpi(p)_k\in (\gamma_{0,k}^x)^{-1}\left(\mathbb{V}^0_x\right) \subset\Gamma^k_x$ be the point determined by the placement of the path $p_1p_2\cdots p_k\in E_{r(p_k)}$ in the order given by $\leq_r$ under the bijection described above. By construction, it satisfies $\gamma^x_k(\varpi(p)_k) = \varpi(p)_{k-1}$ for all $k\in\mathbb{N}$.

        The surjectivity of $\varpi$ is straight-forward: by the discussion above about the order-preserving bijections between $E_v$ and $ (\gamma_{0,k}^x)^{-1}\left(\mathbb{V}^0_x\right)\cap e_v\setminus [\partial^+ e_v]$, for any $z_k\in (\gamma_{0,k}^x)^{-1}\left(\mathbb{V}^0_x\right)$ there is a path $p\in X_B$ such that $\varpi(p)_k = z_k$. Suppose that $p,q\in X_B$ satisfy $\varpi(p) = \varpi(q)$. Then $\varpi(p)_k=\varpi(q)_k$ for all $k$ and, since the coordinates $\varpi(p)_k$ and $\varpi(q)_k$ are determined by the position of a preimage in $(\gamma_{0,k}^x)^{-1}\left(\mathbb{V}^0_x\right)$ for each $k$, it follows that the finite paths $p_1p_2\cdots p_k$ and $q_1q_2\cdots q_k$ are the same for all $k$, i.e. $p=q\in X_B$ and $\varpi$ is injective, and thus a bijection. The proof of injectivity also implies that $\varpi$ is continuous. Now consider a finite path $p'=p_1p_2\cdots p_k \in E_{0k}$ and the corresponding cylinder set $C_{p'}\in X_B$. A point $z_k(p')\in (\gamma_{0,k}^x)^{-1}\left(\mathbb{V}^0_x\right)\cap e_{r(p_k)}\setminus [\partial^+ e_{r(p_k)}]$ is uniquely defined by this path, and thus the image of the cylinder set $C_{p'}$ under $\varpi$ is the clopen set $C_k^\perp(z)$ for any $z\in \varpi(C_{p'})$. It follows that $\varpi$ is a homeomorphism. It remains to show that this map conjugates the actions as desired.

        Let $p\in X_B$ with $s(p) = v\in V_0$ and $p\neq x_{max}$. Then $\varphi_{l_v}((p,0)) = (p,l_v)$ is equivalent to $(\phi(p),0)$ in $\hat{X}_B$. On the $\mho_x\subset \Omega_x$ side, $\varpi((p,0))_0$ is a vertex in $\Gamma^0_x$. Let $k(p)$ be the smallest index so that $p_k$ is not a maximal edge. By construction, $\partial^+e_{v}\sim \partial^-e_{s(p')}$, where $p'$ is the unique path from $V_0$ to $V_k$ with $r(p') = r(p_k)$ and $p'$ is the successor of $p_1p_2\cdots p_k$ in $E_{r(p_k)}$. Thus, by (\ref{eqn:translation}), if $s\in (0,l_v)$ then $\varphi_{s}(\varpi((p,0)))_0 = \varpi((p,0))+s \in e_v\setminus \partial^\pm e_v$ and 
        \begin{equation*}
          \begin{split}
            \varphi_{l_{v}}(\varpi((p,0)))_0 &= \varpi((p,0)) + l_{v} = \partial^+e_{v}\sim \partial^-e_{s(p')} = \varpi((\phi(p),0))_0 = \varpi((p,l_{v}))\\
            &= \varpi(\varphi_{l_{v}}((p,0)))_0.
          \end{split}
        \end{equation*}
        If $p = x_{max}$ then a similar calculation shows the conjugacy between the flows.

        Now let $p,q\in X_B$ with $q\in C_k(p)\setminus C_{k+1}(p)$, that is, with $q_i=p_i$ for $i\leq k$ but $q_{k+1}\neq p_{k+1}$. Then by (\ref{eqn:metricCantor}), $d_{\lambda}(p,q) = \lambda^{-k}$. In addition, by definition, $\varpi(q)_i = \varpi(p)_i$ for all $i<k$ and $\varpi(q)_k \neq \varpi(p)_k$. Thus, by (\ref{eqn:fullMetric}),
        $$\bar{d}_\lambda(\varpi(p),\varpi(q)) = \sum_{i\geq 0}\lambda^{-i}\frac{d_i(\varpi(p)_i,\varpi(q)_i)}{\mathrm{diam}(\Gamma_x^i)} = \sum_{i\geq k}\lambda^{-i}\frac{d_i(\varpi(p)_i,\varpi(q)_i)}{\mathrm{diam}(\Gamma_x^i)}\leq \frac{\lambda^{-k}}{\lambda-1} = \frac{d_\lambda(p,q)}{\lambda-1}.$$
        To obtain a lower bound for $\bar{d}_\lambda(\varpi(p),\varpi(q))$, in view of (\ref{eqn:fullMetric}), it suffices to give a lower bound for $d_k(\varpi(p)_k,\varpi(q)_k)$. By definition, $d_k(\varpi(p)_k,\varpi(q)_k)$ is an minimum the shortest distance between two points in the preimage of $(\gamma^x_{k})^{-1}(\mathbb{V}^{k-1}_x)$. By construction, this is exactly
        $$\min_{v\in V_{k-1}} \lambda_\mu^{-1} \lambda_\mu^{-k+1}(\mathcal{A}_{(k-1)}l)_v.$$
        By standing assumptions, for all $\varepsilon\in(0,1)$ there is a $C_{\varepsilon,B}>1$ such that
        $$C_{\varepsilon,B}^{-1}\lambda_\mu^{(1-\varepsilon)k} \leq (\mathcal{A}_{(k)}l)_v\leq C_{\varepsilon,B} \lambda^{(1+\varepsilon)k}_\mu$$
        for all $k\in\mathbb{N}$ and $v\in V_k$, and thus $C_{\varepsilon,B}^{-1}\lambda_\mu^{-1} \lambda_\mu^{-\varepsilon(k-1)} = K_{\varepsilon,\mu,B} \lambda^{-\varepsilon k}\leq d_k(\varpi(p)_k,\varpi(q)_k)$. Likewise, by standing assumptions, for all $\varepsilon\in (0,1)$ there is a $C'_{\varepsilon,B}$ such that
        \begin{equation}
          \label{eqn:diameter2}
          (C_{\varepsilon,B}')^{-1}\lambda_\mu^{-\varepsilon k} \leq \mathrm{diam}(\Gamma^k_x)\leq C'_{\varepsilon,B} \lambda^{\varepsilon k}_\mu.
        \end{equation}
        As such, for any $\varepsilon\in(0,1)$ there is a $K'_{\varepsilon, B, \mu}$ such that
        $$\frac{d_k(\varpi(p)_k,\varpi(q)_k)}{\mathrm{diam}(\Gamma^k_x)}\geq K'_{\varepsilon,\mu, B}\lambda_\mu^{-\varepsilon k} = K'_{\varepsilon,\mu, B}d_{\lambda_\mu}(p,q)^\varepsilon$$
        for all $k\in \mathbb{N}$. Putting this together with the upper bound above,
        $$K'_{\varepsilon,\mu, B}d_{\lambda_\mu}(p,q)^{1+\varepsilon}\leq \bar{d}_{\lambda_\mu}(\varpi(p),\varpi(q)) \leq \frac{d_{\lambda_\mu}(p,q)}{\lambda-1}.$$
      \end{proof}

      Let $\mu_x$ be the unique, $\mathbb{R}$-invariant ergodic probability measure on $\Omega_x$. By the local product structure of $\Omega_x$, the measure is locally of the form $\mathrm{Leb}\times \nu$, where $\nu$ is a transverse invariant measure in the sense of \cite{BM:UE}. Denote by $\nu_{k,z}$ the transverse invariant measure $\nu$ restricted to $C^\perp_k(z)$. For any $k\in \mathbb{N}_0$ and $z\in \Omega_x$, set
      $$\hat{\nu}_{k,z}:= \nu(C^\perp_k(z)) = \int_{C^\perp_k(z)} \, d\nu_{k,z}.$$
      Note that by invariance, for all $|s|$ small enough:
      $$\hat{\nu}_{k,z} = \nu(C^\perp_k(z)) = \nu(\varphi_{-s}(C^\perp_k(\varphi_s(z)))) = (\varphi_{-s})_*\nu(C^\perp_k(\varphi_s(z))) = \hat\nu_{k,\varphi_s(z)}.$$
	Although this will not be used, under the standing assumptions, it can be shown that for all $\varepsilon>0$ there is a $C_\varepsilon>1$ such that for all $k\in\mathbb{N}_0$ and $z\in\Omega_x$:
        $$C_\varepsilon^{-1}\lambda_\mu^{-k-\varepsilon} \leq \hat\nu_{k,z} \leq C_\varepsilon \lambda_\mu^{-k+\varepsilon}.$$

      
      \subsection{Function spaces}
      Let $\mathcal{A}$ be the Borel $\sigma$-algebra of $\Omega_x$, and for any $k\geq 0$ let $\mathcal{A}_k\subset \mathcal{A}$ be $\sigma$-algebras generated by sets of the form $(\pi^x_k)^{-1}(A)$ for open balls $A\subset \Gamma_x^k$. For any measurable $f:\Omega_x\rightarrow \mathbb{R}$, let $\Pi_k f:\Omega_x\rightarrow \mathbb{R}$ be the conditional expectation of $f$ on $\mathcal{A}_k$. This is expressed explicitly as
      \begin{equation}
        \label{eqn:Pi}
        \Pi_kf(z) = \hat{\nu}^{-1}_{k,z}\int_{C_k^\perp(z)} f(y)\, d\nu_{k,z}(y).
      \end{equation}
      The function $\Pi_kf$ is transversally-locally constant: if $y,z$ agree on the first $n>k$ coordinates, $C_k^\perp(x) = C_k^\perp(y)$ and $\Pi_kf(z) = \Pi_kf(y)$. In particular, $\Pi_kf(z)$ depends only on the first $k+1$ coordinates $z_0,z_1,\dots, z_k$.

      Let $\delta_kf:= \Pi_kf - \Pi_{k-1}f$. Then
      $$\Pi_kf = \sum_{i=0}^k \delta_i f \hspace{.6in} \mbox{ and so } \hspace{.6in}\lim_{k\rightarrow \infty}  \Pi_k f= \lim_{k\rightarrow \infty} \sum_{i=0}^k \delta_i f  = f$$
      in $L^1$ by the increasing martingale theorem. Thus, any measurable funtion $f$ has the unique decomposition
      \begin{equation}
        \label{eqn:expansion}
        f = \sum_{k\geq 0} \delta_k f.
      \end{equation}
      Each function $\delta_k f$ depends only on the coordinates $z_k$. As such, there is a function $g_k:\Gamma_x^k\rightarrow \mathbb{R}$ such that $\delta_kf(z) = \pi_k^{*}g_k(z)=g_k(z_k)$. Denote this function by $\pi_{k*}^x\delta_k f$. 

      Since $\Omega_x$ carries a natural $\mathbb{R}$ action, let $X$ be the generating vector field of this action. For a function $f:\Omega_x\rightarrow \mathbb{R}$, define the function $Xf$ to be the function with value at $z\in \Omega_x$:
      $$Xf(z) := \lim_{s\rightarrow 0}\frac{f\circ\varphi_s(z) - f(z)}{s}$$
      if the limit exists. Let $C^1(\Omega_x)$ be the space of functions such that $Xf$ exists and is continuous. The spaces $C^r(\Omega_x)$ are defined recursively for $r>1$.
      
      Let $C^r(\Gamma_x^k)$ and $C^\infty(\Gamma_x^k)$ be the spaces of the corresponding regularity, and define
      $$C^{\infty}_{tlc}(\Omega_x) := \bigcup_{k\geq 0}\pi^{x*}_k C^\infty\left(\Gamma_x^k\right).$$
      For $r\geq 0$, $\alpha>0$, and $\varepsilon\in(0,\alpha)$ define the norms on $C^{\infty}_{tlc}(\Omega)$ as
      \begin{equation}
        \label{eqn:norm}
        \|f\|_{r,\alpha} = \sum_{k\geq 0} \lambda_\mu^{\alpha k} \|\pi_{k*}^x\delta_k f\|_{C^r(\Gamma_x^k)}\hspace{.25in}\mbox{ and }\hspace{.25in}\vertiii{f}_{r,\alpha,\varepsilon} = \sup_{k\geq 0} \lambda_\mu^{k(\alpha-\varepsilon)} \|\pi_{k*}^x\delta_k f\|_{C^r(\Gamma_x^k)}.
      \end{equation}
      Both of these norms are finite for any $f\in C^\infty_{tlc}(\Omega)$ since $\pi_{k*}^x\delta_k f = 0$ for all $k$ large enough. Since any measurable function can be written as in (\ref{eqn:expansion}), the first norm in (\ref{eqn:norm}) captures the decay rate of each of the components of the function as the components go through all the complexes $\Gamma_x^k$ used in the inverse limit.

      Define
      $$\mathcal{S}_\alpha^r(\Omega_x) := \overline{C^\infty_{tlc}(\Omega_x)}^{\|\cdot\|_{r,\alpha}}\hspace{.6in}\mbox{ and }\hspace{.6in} \mathcal{R}_{\alpha,\varepsilon}^r(\Omega_x) := \overline{C^\infty_{tlc}(\Omega_x)}^{\vertiii{\cdot}_{r,\alpha,\varepsilon}}$$
      to be the respective completions of $C^\infty_{tlc}(\Omega_x)$ under the norms in (\ref{eqn:norm}). 

      \begin{proposition}
        \label{eqn:spaceProps}
        For $r\in\mathbb{N}$ and $\alpha>1$:
        \begin{enumerate}
        \item For any $k\in\mathbb{N}$, $[X,\delta_k] = 0$ on $C^1(\Omega_x)$
        \item(Boundedness of differential operator)  The leafwise differential operator $X:\mathcal{S}_\alpha^r(\Omega_x)\rightarrow \mathcal{S}_{\alpha+1}^{r-1}(\Omega_x)$ satisfies $\|Xf\|_{r-1,\alpha+1}\leq \|f\|_{r,\alpha} $
        \item(Poincar\'e-ish inequality) For all $\varepsilon>0$ there is a $C_{\mu,\varepsilon}$ such that for $r\geq 0$ and $f\in \mathcal{S}^{r+1}_\alpha(\Omega_x)$,
          $$\|f-\mu_x(f)\|_{r,\alpha}\leq C_{\mu,\varepsilon} \|Xf\|_{r,\alpha+1+\varepsilon}.$$
        \item For all $\varepsilon\in(0,\alpha)$ there is a continuous embedding $i:\mathcal{S}^r_\alpha(\Omega_x)\rightarrow \mathcal{R}_{\alpha,\varepsilon}^r(\Omega_x)$ satisfying $\vertiii{i(f)}_{r,\alpha,\varepsilon}\leq \|f\|_{r,\alpha}.$
        \end{enumerate}
      \end{proposition}
      \begin{proof}
        First, by (\ref{eqn:Pi}), for $f\in C^1(\Omega_x)$:
        \begin{equation}
          \label{eqn:commutant}
          \begin{split}
            X\Pi_kf(z) &= \lim_{s\rightarrow 0}\frac{\Pi_kf(\varphi_s(z))- \Pi_kf(z)}{s}\\
            &= \lim_{s\rightarrow 0}\frac{\displaystyle\hat{\nu}^{-1}_{k,\varphi_s(z)} \int_{C^\perp_k(\varphi_s(z))} f(y)\, d\nu_{k,\varphi_s(z)}(y) -  \hat{\nu}^{-1}_{k,z} \int_{C^\perp_k(z)} f(y)\, d\nu_{k,z}(y) }{s}\\
            &= \lim_{s\rightarrow 0}\hat{\nu}^{-1}_{k,z} \int_{C^\perp_k(z)} \frac{  f(\varphi_s(y)) -  f(y) }{s} \, d\nu_{k,z}(y)\\
            &= \Pi_k X f(z),
          \end{split}
        \end{equation}
        which proves (i).
        
        For (ii), for $f\in\mathcal{S}^r_{\alpha}$, from (\ref{eqn:translation}) and (\ref{eqn:commutant}),
        \begin{equation}
          \begin{split}
            \delta_k X f(z) = X\delta_k f(z) &= \lim_{s\rightarrow 0} \frac{\delta_kf\circ\varphi_s(z) - \delta_kf(z)}{s} = \lim_{s\rightarrow 0} \frac{\pi_{k*}\delta_kf(z_k+\lambda^{-k}_\mu s) - \pi_{k*}\delta_kf(z_k)}{s}\\
            &= \lambda_\mu^{-k} (\pi_{k*}\delta_k f)'(z_k),
          \end{split}
        \end{equation}
        and so
        \begin{equation}
          \label{eqn:pass}
          \pi_{k*}\delta_k Xf = \lambda^{-k}_\mu(\pi_{k*}\delta_k f)',
        \end{equation}
        from which it follows that
        $$\| \pi_{k*}\delta_k Xf \|_{C^{r-1}(\Gamma_x^k)} \leq  \lambda^{-k}_\mu\|\pi_{k*}\delta_k f \|_{C^{r}(\Gamma_x^k)},$$
        and thus
        \begin{equation*}
          \begin{split}
            \|Xf\|_{r-1,\alpha+1} = \sum_{k\geq 0 }\lambda_\mu^{(\alpha+1)k} \| \pi_{k*}\delta_k Xf \|_{C^{r-1}(\Gamma_x^k)} \leq \sum_{k\geq 0 }\lambda_\mu^{(\alpha+1)k} \lambda^{-k}_\mu\| \pi_{k*}\delta_k f \|_{C^{r}(\Gamma_x^k)}= \|f\|_{r,\alpha}
          \end{split}
        \end{equation*}
        which proves the second statement.

        To prove the third inequality, a few basic observations are in order. Let $u\in C^1(\Gamma^k_x)$. Then by continuity there exists a $p_u\in \Gamma^k_x$ such that $u(p_u) = \mu_k(u)$, where $\mu_k$ is the Lebesgue measure on $\Gamma^k_x$. Thus for any $y\in \Gamma^k_x$,
        $$u(y) - \mu_k(u) = u(y) - u(p_u) = \int_{p_u}^y u'(q)\, d\mu_k(q)$$
        and so it follow that
                $$\|u-\mu_k(u)\|_\infty \leq \|u'\|_{\infty}\cdot \mathrm{diam}(\Gamma^k_x).$$

        Take $u = \pi_{k*}\delta_k \Theta$ for some $\Theta\in\mathcal{S}^{r+1}_\alpha$. By the bound above and (\ref{eqn:pass}):
        $$\lambda^{k}_\mu\|\pi_{k*}\delta_k X\Theta\|_\infty = \|(\pi_{k*}\delta_k \Theta)'\|_\infty\geq  \|\pi_{k*}\delta_k \Theta - \mu_k(\pi_{k*}\delta_k \Theta)\|_\infty \cdot \mathrm{diam}(\Gamma^k_x)^{-1}.$$
or
$$\|\pi_{k*}\delta_k \Theta - \mu_k(\pi_{k*}\delta_k \Theta)\|_\infty\leq \lambda^{k}_\mu\|\pi_{k*}\delta_k X\Theta\|_\infty \cdot \mathrm{diam}(\Gamma^k_x) .$$   

It should be noted that for $k>0$, $\mu_k(\pi_{k*}\delta_k \Theta) = 0$. Indeed, by construction, $\mu_k(\pi_{k*}\delta_k \Theta)  = \mu(\delta_k\Theta)$, and since the average of $\Theta$ contributes only to $\Pi_0\Theta$, $\delta_k \Theta$ has zero average.

        To apply to higher derivatives, first, by the basic Poincar\'e inequality for $\Gamma^k_x$ derived above,
        $$\|(\pi_{k*}\delta_k \Theta)^{(j)}\|_\infty\leq  \|(\pi_{k*}\delta_k  \Theta)^{(j+1)}\|_{\infty}\cdot \mathrm{diam}(\Gamma^k_x).$$
        Combining this with $(\pi_{k*}\delta_k X \Theta)^{(j)} = \lambda^{-k}_\mu(\pi_{k*}\delta_k \Theta)^{(j+1)}$ for\footnote{Note that this is the place were it is used that $\Theta\in\mathcal{S}^{r+1}_\alpha$ and not merely in $\mathcal{S}^r_\alpha$.} $j<r$, which follows from (\ref{eqn:pass}),  gives
        $$\|(\pi_{k*}\delta_k \Theta)^{(j)}\|_\infty\leq \lambda_\mu^{ k} \|(\pi_{k*}\delta_k X\Theta)^{(j)}\|_\infty \cdot \mathrm{diam}(\Gamma^k_x).$$
        Putting everything together:
        \begin{equation}
          \begin{split}
            \|\Theta  - \mu(\Theta)\|_{r,\alpha} &= \sum_{k\geq 0 }\lambda_\mu^{\alpha k} \|\pi_{k*}\delta_k (\Theta-\mu(\Theta)) \|_{C^r(\Gamma_x^k)} \\
            &\leq \sum_{k\geq 0 }\lambda_\mu^{\alpha k} \lambda_\mu^{k} \|\pi_{k*}\delta_k X \Theta \|_{C^r(\Gamma_x^k)}\cdot \mathrm{diam}(\Gamma^k_x) \\
            &\leq \sum_{k\geq 0 } C_{\varepsilon,x}\cdot \lambda_\mu^{(\alpha+1+\varepsilon) k} \|\pi_{k*}\delta_k X \Theta \|_{C^r(\Gamma_x^k)} \\
            &= C_\varepsilon \|X \Theta\|_{r,\alpha+1+\varepsilon},
          \end{split}
        \end{equation}
        which proves (iii).

        For (iv), let $f\in\mathcal{S}_{\alpha}^r$. Then for any $\varepsilon\in(0,\alpha)$, for all but finitely many indices $k\in\mathbb{N}$,
        $$\|\pi_{k*}\delta_k f\|_{C^r}\leq \lambda_\mu^{-k(\alpha-\varepsilon)}.$$
        Indeed, let $I_\varepsilon\subset\mathbb{N}$ the the set of $k$ such that $\|\pi_{k*}\delta_k f\|_{C^r}> \lambda_\mu^{-k(\alpha-\varepsilon)}$ and suppose that $|I_\varepsilon|$ is not finite. Then
      $$\|f\|_{r,\alpha} = \sum_{k\geq 0}\lambda_\mu^{k\alpha}\|\pi_{k*}\delta_k f\|_{C^r}\geq \sum_{k\in I_\varepsilon}\lambda_\mu^{k\alpha}\|\pi_{k*}\delta_k f\|_{C^r} \geq\sum_{k\in I_\varepsilon}\lambda_\mu^{k\varepsilon} = \infty,$$
        a contradiction. Thus the sup in the definition of $\vertiii{\cdot}_{r,\alpha, \varepsilon}$ is obtained at a specific index, say $k'$. Thus
        $$\vertiii{f}_{r,\alpha,\varepsilon} = \lambda^{k'(\alpha-\varepsilon)}_\mu \|\pi_{k'*}\delta_{k'} f \|_{C^r(\Omega_x)}\leq \sum_{k\geq 0} \lambda^{k\alpha}_\mu \|\pi_{k*}\delta_k f \|_{C^r(\Omega_x)}  = \|f\|_{r,\alpha}.$$
        In particular this shows that the norm of the inclusion operator is bounded by 1, and is actually 1 when $f = 1$. So the embedding is continuous.
      \end{proof}
      \section{The cohomological equation}
      \label{eqn:coh}
      \begin{definition}
        For $r\in\mathbb{R}$, $\alpha>1$ and $\varepsilon\in(0,\alpha-1)$, the $r,\alpha,\varepsilon$ cohomology of $\Omega_x$ is the quotient
        $$H_{r,\alpha,\varepsilon}^1(\Omega_x) :=  \mathcal{S}^r_\alpha(\Omega_x)/X\mathcal{S}^{r+1}_{\alpha-1-\varepsilon}(\Omega_x).$$
      \end{definition}
      The purpose of this section is to relate $H_{r,\alpha,\varepsilon}^1(\Omega_x)$ to $\check H^1(\Omega_x;\mathbb{R})$. To do this, first let
      $$C^r_{tlc}(\Omega_x) := \bigcup_{k\geq 0}\pi^*_k C^r(\Gamma^k)$$
      and define $H^1_{r,tlc}(\Omega_x):= C^r_{tlc}(\Omega_x)/XC^{r}_{tlc}(\Omega_x)$. Finally, let $H_{tlc}^1(\Omega_x):= C^\infty_{tlc}(\Omega_x)/XC^\infty_{tlc}(\Omega_x)$. It follows from \cite{sadun:deRham} that $H^1_{tlc}(\Omega_x) \cong \check H^1(\Omega_x;\mathbb{R})$.
      \begin{theorem}
        \label{thm:r-alpha}
        Let $r\in\mathbb{N}$, $\alpha>1$ and $f\in\mathcal{S}^r_\alpha(\Omega_x)$. For every $\varepsilon\in(0,\alpha-1)$ there exists a $g\in C^r_{tlc}(\Omega_x)$ and $u\in \mathcal{S}^{r+1}_{\alpha-1-\varepsilon}(\Omega_x)$ such that $f-g = Xu$.
      \end{theorem}
      \begin{proof}[Proof of Theorem \ref{thm:r-alpha}]
        First it needs to be shown that for $r\in\mathbb{N}$, $H^1_{r,tlc}(\Omega_x) \cong \check H^1(\Omega_x;\mathbb{R})$. Indeed, let $f\in C^r_{tlc}(\Omega_x)$. As shown in \cite[Proposition 3.1]{T:transversal}, de Rham regularization can be used to show that there is a $f'\in C^\infty_{tlc}(\Omega_x)$ and $g\in C^r_{tlc}(\Omega_x)$ such that $f'-f = Xg$. Thus, $H^1_{r,tlc}(\Omega_x) \cong H^1_{tlc}(\Omega_x) \cong \check H^1(\Omega_x;\mathbb{R})$.

        Now let $f\in\mathcal{S}^r_\alpha(\Omega_x)$ for $r\in\mathbb{N}$ and $\alpha>1$. The task is, for any $\varepsilon\in(0,\alpha-1)$, to find a representative $g$ of $[g] \in H^1_{r,tlc}(\Omega_x)$ such that $f-g = Xu$ for some $u\in \mathcal{S}^{r+1}_{\alpha-1-\varepsilon}$. In view of the decomposition $f = \sum_k \delta_k f$, define $f_n := \sum_{k=0}^n \delta_k f$ for every $n$. Note that each $f_n$ has a class $[f_n]\in H^1_{r,tlc}(\Omega_x)$.
        The sequence $\{[f_n]\}$ of classes converges in $H^1_{r,tlc}(\Omega_x)$. Indeed, for $n>m>0$:
        $$\left\|[f_n]-[f_m]\right\|\leq \left\|\sum_{k>m}[\delta_k f] \right\| \leq K_x\sum_{k>m} \|\pi_{k*}\delta_k f\|_{C^r(\Gamma^k)}\leq K_{x,f,\varepsilon} \lambda_\mu^{-m(\alpha-\varepsilon)},$$
        for all $\varepsilon\in(0,\alpha)$, where \cite[Lemma 3.1]{T:transversal} was used in the second inequality. Denote by $[f_*]$ the limit of $[f_n]$ in $H^1_{r,tlc}(\Omega_x)$. Note that $f_n\rightarrow f$ in $\mathcal{S}^r_\alpha(\Omega_x)$, and so the assignment of $[f_*]$ to $[f]\in H^1_{r,\alpha}(\Omega_x)$ needs to be justified.

        Since $\check H^1(\Omega_x;\mathbb{R}) $ is a direct limit, there is a $k'\in\mathbb{N}$ such that every class $c\in \check H^1(\Omega_x;\mathbb{R})$ is determined by a class $c\in \check H^1(\Gamma^{k'}_x;\mathbb{R})$. Indeed, by ergodicity, $\beta_x:= \dim H^1(\Omega_x;\mathbb{R})<\infty$ is constant $\mu$-almost everywhere, and so there is a $k'$ and a subspace $W\subset H^1(\Gamma^{k'}_x;\mathbb{R})$ of dimension $\beta_x$ which contains a representative of every class in $H^1(\Omega_x;\mathbb{R})$.

        By \cite{sadun:deRham}, $\check H^1(\Gamma^{k'}_x;\mathbb{R})\cong C^\infty(\Gamma^{k'}_x)/XC^\infty(\Gamma^{k'}_x)$, and so any class $c\in\check H^1(\Gamma^{k'}_x;\mathbb{R})$ is determined by some $f_c\in C^\infty(\Gamma^{k'}_x)$ through $\pi^*_{k'}f_c$. Thus, for every $n$ there exists $g_n\in C^r(\Gamma_x^{k'})$ and $u_m\in C^r(\Gamma_x^{m})$, with $m = \max\{k',n\}$, such that $\delta_nf - \pi^*_{k'}g_n = X \pi^*_mu_m$. Without loss of generality, the functions $u_m$ are taken to be of zero average. Note that for $n>k'$,
        $$ \delta_n X \pi^*_nu_n = \delta_n(\delta_nf-\pi^*_{k'}g_n) = \delta_nf,$$
        and so by (\ref{eqn:pass}), $\pi_{n*} \delta_n X \pi^*_n u_n  = \lambda^{-n}_\mu (\pi_{n*}\delta_n \pi^*_n u_n )' =  \lambda^{-n}_\mu u_n' = \pi_{n*} \delta_n f  $ for $n>k'$. Since $f\in\mathcal{S}^r_\alpha(\Omega_x)$, this implies that $u_n\in C^{r+1}(\Gamma^n_x)$. In addition
        \begin{equation}
          \label{eqn:uBnd1}
          \|u_n^{(j)}\|_{\infty} = \lambda^{n}_\mu \|(\pi_{n*}\delta_n f)^{(j-1)} \|_\infty \leq  \lambda^{n}_\mu \|\pi_{n*}\delta_n f \|_{C^r(\Gamma^n_x)} \leq \lambda^{-n(\alpha-1-\varepsilon)}_\mu \vertiii{f}_{r,\alpha,\varepsilon}
        \end{equation}
        for $j = 1, \dots , r+1$. Moreover, for $\varepsilon\in(0,(\alpha-1)/2)$, using (\ref{eqn:diameter2}) and (\ref{eqn:pass}), it follows that
        \begin{equation}
          \label{eqn:uBnd2}
          \begin{split}
            \|u_n\|_{\infty} &\leq  \|u'_n\|_{\infty} \cdot \mathrm{diam}(\Gamma^k_x)\leq C_{\varepsilon,x}\lambda^n_\mu \|\pi_{n*} \delta_n f \|_{\infty} \cdot \mathrm{diam}(\Gamma^k_x) \\
            &\leq C_{\varepsilon,x}' \vertiii{f}_{r,\alpha,\varepsilon} \lambda^{-n(\alpha-1-2\varepsilon)}_\mu.
          \end{split}
        \end{equation}
        Thus (\ref{eqn:uBnd1}) and (\ref{eqn:uBnd2}) imply that there exists a $K>0$, independent of $n$, such that
        $$\|u_n\|_{C^r(\Gamma^n_x)}\leq K \lambda^{-n(\alpha-1-2\varepsilon)}.$$
        Thus, letting $U_n := \sum_{k=0}^n \pi^*_ku_k$ for $n\in\mathbb{N}$, it follows that
        $U_\infty:= \lim_n U_n \in \mathcal{S}^{r+1}_{\alpha-1-\varepsilon}$ for any $\varepsilon$. Moreover, by construction,
        $$\sum_{k=0}^n\delta_k f - \pi^*_{k'}\sum_{k=0}^ng_k = X\sum_{k=0}^n \pi^*_k u_k$$
        for every $n$, and so taking limits, it follows that $f - g = XU_\infty$ for $g\in C^r_{tlc}(\Omega_x)$.
      \end{proof}
      \begin{corollary}
        \label{cor:coh}
        For $r\in\mathbb{N}$, $\alpha>1$ and $\varepsilon\in(0,\alpha-1)$,
        $$H^1_{r,\alpha,\varepsilon}(\Omega_x)\cong \check H^1(\Omega;\mathbb{R}).$$
      \end{corollary}
      \subsection{Tame estimates}
      Let $x = (B,\leq_r)$ satisfy the standing assumptions and let $\tau_x = \min_{e\in \mathbb{E}^0_x}|e_v|$ be length of the smallest edge in $\Gamma^0_x$. Let $u_\tau:\mathbb{R}\rightarrow \mathbb{R}$ be a smooth bump function of integral one compactly supported on the interval $(-\frac{\tau_x}{5},\frac{\tau_x}{5})$. For any function $h:\mho_x\rightarrow \mathbb{R}$, the function $h_\tau$ will denote the ``bumpification'' of $h$ in $\Omega$. More precisely, the $\tau_x/5$-neighborhood of $\mathbb{V}^0_x$ can be lifted to $\Omega_x$ as $\mathcal{U}_x := \pi^{-1}_0(B_{\frac{\tau_x}{5}}(\mathbb{V}^0_x))$. As such, $\mathcal{U}_x$ admits the coordinates $(t,c)$, for $t\in [-\frac{\tau_x}{5},\frac{\tau_x}{5}]$ and $c\in \mho_x$. In these coordinates, $h_\tau(t,c) = u(t)h(c)$. Likewise, using the map $\varpi:X_B\rightarrow \mho_x$ from Proposition \ref{prop:almostBiLip}, functions on $X_B$ can be brought over to $\mho_x$, and then bumpified. The question is what happens to an $\alpha$-H\"older function once it is brought over to $\Omega_x$ using $\varpi$ and then bumpified. The following lemma gives the answer.
      \begin{lemma}
        \label{lem:bumpy}
        For $h\in H_\alpha(X^+_\mathcal{B})$ define $\hat{h}_\tau := ((\varpi^{-1})^*h)_\tau:\Omega_x\rightarrow \mathbb{R}$ be the corresponding function supported on $\mathcal{U}_x\subset \Omega_x $ as described above. Then for any $r\in \mathbb{N}$ and $\beta>0$ such that $\beta+r<\alpha$, $\hat{h}_\tau\in\mathcal{S}_{\beta}^r(\Omega_x)$ with
        $$\|\hat{h}_\tau\|_{r,\beta} \leq \frac{2r\lambda_\mu^\alpha \|u_\tau\|_{C^r}}{(\lambda_\mu-1)(1-\lambda_\mu^{\beta+r-\alpha})} \|h\|_\alpha.$$
      \end{lemma}
      \begin{proof}
        Using local coordinates $z = (t,c) \in \mathcal{U}_x$, by definition,
        $$\delta_k \hat{h}_\tau(t,c) = u(t)\left[  \hat{\nu}_{k,z}^{-1}\int_{C^\perp_k(z)} h(\varpi^{-1}(c')) \, d\nu_{k,z}- \hat{\nu}_{k-1,z}^{-1}\int_{C^\perp_{k-1}(z)} h(\varpi^{-1}(c')) \, d\nu_{k-1,z}\right] = \pi^*_k g_k(z)$$
        for some $g_k:\Gamma^k_x\rightarrow \mathbb{R}$ which is supported on the $\lambda_\mu^{-k}\tau_x/5$-neighborhood of $(\gamma_{0,k}^x)^{-1}(\mathbb{V}^0_x)$. Indeed, in a local coordinate $y$ centered around a point $z_k\in (\gamma_{0,k}^x)^{-1}(\mathbb{V}^0_x) $, the function $g_k$ is of the form $g_k(y) = u(\lambda_\mu^k y)h_k(z_k)$, where $h_k(z_k) = \Pi_k [(\varpi^{-1})^*h](z_*) - \Pi_{k-1}[(\varpi^{-1})^*h](z_*)$ for some $z_*\in \pi^{-1}_k(z_k)$. As such, $\|g_k^{(r)}\|_\infty \leq  \lambda_\mu^{rk}\|u^{(r)}\|_\infty |h_k(z_k)|$, and so it remains to bound $|h_k(z_k)|$.

        To bound this last term,
        \begin{equation*}
          \begin{split}
            h_k(z_k) &=  \hat{\nu}_{k,z_*}^{-1}\int_{C^\perp_k(z_*)} h(\varpi^{-1}(c')) \, d\nu_{k,z_*}- \hat{\nu}_{k-1,z_*}^{-1}\int_{C^\perp_{k-1}(z_*)} h(\varpi^{-1}(c')) \, d\nu_{k-1,z_*} \\
            &=  \hat{\nu}_{k,z_*}^{-1}\int_{C^\perp_k(z_*)} h(\varpi^{-1}(c')) - h(\varpi^{-1}(c)) \, d\nu_{k,z_*} \\ & \hspace{1.5in}- \hat{\nu}_{k-1,z_*}^{-1}\int_{C^\perp_{k-1}(z_*)} h(\varpi^{-1}(c')) -h(\varpi^{-1}(c))\, d\nu_{k-1,z_*} \\            
          \end{split}
        \end{equation*}
        and so by (\ref{eqn:diameter})
        $$|h_k(z_k)|\leq \mathrm{diam}(C_k^\perp(z_*))^\alpha|h|_\alpha + \mathrm{diam}(C_{k-1}^\perp(z_*))^\alpha|h|_\alpha\leq \frac{2\lambda_\mu^\alpha}{\lambda_\mu-1} \lambda_\mu^{-k\alpha}|h|_\alpha, $$
        from which it follows that
        $$\|g_k^{(r)}\|_\infty \leq \frac{2\lambda_\mu^\alpha \|u^{(r)}\|_\infty}{\lambda_\mu-1} \lambda_\mu^{-k(\alpha-r)}|h|_\alpha .$$
        Thus,
        \begin{equation*}
          \begin{split}
            \|\hat{h}_\tau\|_{r,\beta} &= \sum_{k\geq 0}\lambda_\mu^{\beta k}\|\pi_{k*}\delta_k \hat{h}_\tau\|_{C^r(\Gamma^k_x)}\leq  \sum_{k\geq 0} \frac{2r\lambda_\mu^\alpha \|u\|_{C^r}}{\lambda_\mu-1} \lambda_\mu^{(\beta+r-\alpha)k}|h|_\alpha \\
            &\leq \frac{2r\lambda_\mu^\alpha \|u\|_{C^r}}{(\lambda_\mu-1)(1-\lambda_\mu^{\beta+r-\alpha})} |h|_\alpha
          \end{split}
        \end{equation*}
        if $\beta+r<\alpha$.
      \end{proof}
      In light of lemma \ref{lem:bumpy} -- which showed that if $h\in H_\alpha(X^+_\mathcal{B})$ then $\hat{h}_\tau\in\mathcal{S}_{\alpha-1-\varepsilon}^1$ -- and Corollary \ref{cor:coh}, the function $\hat{h}_\tau$ has a class $[\hat{h}_\tau]\in H^1_{1,\alpha-1-\varepsilon,\varepsilon'}(\Omega_x)\cong \check H^1(\Omega_x;\mathbb{R})$ as long as $\alpha-1-\varepsilon-\varepsilon'>1$. Since $\check H^1(\Omega_x;\mathbb{R})$ is finite dimensional, then $[\hat{h}_\tau]$ is contained in a finite dimensional vector space as long as $\alpha>2$. What remains to show in this section is that the class $[\hat{h}_\tau]$ is trivial, i.e. $\hat{h}_\tau = X\Theta$ with $\Theta\in \mathcal{S}^2_{\alpha-2-\varepsilon}$ if and only if $h = g\circ \varphi - g$, with estimates on $g$. Before getting to that, one more lemma for estimates is needed.
      \begin{lemma}
        \label{lem:pullback}
        Let $\Theta\in \mathcal{S}^r_\alpha$ for $r\in\mathbb{N}$, $\alpha>0$ and define $g:\mho_x\rightarrow\mathbb{R}$ by the restriction of $\Theta$ to $\mho_x$. Then $\varpi^*g\in H_{\alpha-\varepsilon}(X_B)$ for all $\varepsilon\in (0,\alpha)$. Moreover, for any $\varepsilon\in (0,\alpha)$ there is a $C_{\alpha,\varepsilon,x}$ such that
        $$\|\varpi^*g\|_{\alpha-\varepsilon} \leq C_{\alpha,\varepsilon,x} \|\Theta\|_{r,\alpha}.$$
      \end{lemma}
      \begin{proof}
        Let $\Theta\in \mathcal{S}^r_\alpha$ and $\varepsilon\in (0,\alpha)$. Part (iv) of Proposition  \ref{eqn:spaceProps} implies that $\|\pi_{k*}\delta_k\Theta\|_{C^r}\leq \lambda_\mu^{-k(\alpha-\varepsilon)}\vertiii{\Theta}_{r,\alpha,\varepsilon}$ for all $k\geq 0$.
        
        Let $c_1,c_2\in X_B$ with $c_2\in C_k(c_1)\setminus C_{k+1}(c_1)$, that is, they agree in the first $k$ coordinates but not in the $(k+1)^{st}$. Since $g(c) = \Theta(0,c)$ in the product coordinates in the neighborhood $\mathcal{U}_x$ of $\mho_x\subset\Omega_x$, for $\beta>0$,
        \begin{equation*}
          \begin{split} 
            \frac{|g(\varpi(c_1))-g(\varpi(c_2))|}{d_{\lambda_\mu}(c_1,c_2)^\beta} &= \frac{|g(\varpi(c_1))-g(\varpi(c_2))|}{\lambda_\mu^{-\beta k}} \leq \lambda_\mu^{\beta k}\sum_{n\geq k} |\delta_n\Theta(0,\varpi(c_1))- \delta_n\Theta(0,\varpi(c_2))|  \\
            &\leq \lambda_\mu^{\beta k}\sum_{n\geq k}\|\pi_{n*}\delta_n\Theta\|_{C^r}\cdot \mathrm{diam}(\Gamma^k_x) \leq  \lambda^{\beta k}_\mu \vertiii{\Theta}_{r,\alpha,\varepsilon} C_{x,\varepsilon'}\sum_{n\geq k}\lambda_{\mu}^{-k(\alpha-\varepsilon-\varepsilon')} \\
            &\leq  \lambda^{\beta k}_\mu C_{x,\varepsilon'} \vertiii{\Theta}_{r,\alpha,\varepsilon} \frac{\lambda_{\mu}^{-k(\alpha-\varepsilon-\varepsilon')}}{1-\lambda_\mu^{\varepsilon+\varepsilon'-\alpha}}
          \end{split}
        \end{equation*}
        for any $\varepsilon'\in (0,\alpha-\varepsilon)$, where it was used that for all $\varepsilon'>0$, $\mathrm{diam}(\Gamma^k_x)\leq C_{x,\varepsilon'}\lambda_\mu^{k\varepsilon'}$ for all $k>0$. Thus if $\beta = \alpha-\varepsilon-\varepsilon'$, the bound above gives the bound on the H\"older constant of $g$:
        $$ |\varpi^*g|_{\alpha-\varepsilon-\varepsilon'} \leq \frac{C_{x,\varepsilon'} \vertiii{\Theta}_{r,\alpha,\varepsilon}}{1-\lambda_\mu^{\varepsilon+\varepsilon'-\alpha}}\leq \frac{C_{x,\varepsilon'}}{1-\lambda_\mu^{\varepsilon+\varepsilon'-\alpha}} \|\Theta\|_{r,\alpha}$$
        by part (iv) of Proposition \ref{eqn:spaceProps}. Finally, $\|\varpi^* g\|_\infty \leq \|\Theta\|_\infty \leq \sum_{k\geq 0} \|\pi_{k*}\delta_k\Theta\|_\infty\leq \|\Theta\|_{r,\alpha}$ and so
        $$\|\varpi^*g\|_{\alpha-\varepsilon-\varepsilon'}\leq \left(1+ \frac{C_{x,\varepsilon'} }{1-\lambda_\mu^{\varepsilon+\varepsilon'-\alpha}} \right) \|\Theta\|_{r,\alpha}$$
        which finishes the proof.
      \end{proof}
      The following proposition completes the proof of Theorem \ref{thm:main}.
      \begin{proposition}
        Under the standing hypothesis, for $\mu$-almost every $x = (B,\leq_r)$, for $\alpha>2$, there are finitely many $\phi$-invariant distributions $\mathcal{D}_1,\dots, \mathcal{D}_{\mu}\in H_\alpha(X_B)'$ such that $\mathcal{D}_1(h) = \cdots =\mathcal{D}_{\mu}(h)=0$ if and only if for all $\varepsilon\in (0,\alpha-2)$ there exists a $K = K(\alpha,\varepsilon,x,\mu)$ and $g\in H_{\alpha-2-\varepsilon}(X_B)$ with $h = g\circ \phi-g$ and 
        $$\|g\|_{\alpha-2-\varepsilon}\leq K \|h\|_\alpha.$$
      \end{proposition}
      \begin{proof}
        Let $h\in H_\alpha(X_B)$ for $\alpha>2$ and define $\hat{h}_\tau = ((\varpi^{-1})^*h)_\tau:\Omega_x\rightarrow \mathbb{R}$ be the bumpification supported compactly on $\mathcal{U}_x$ as in Lemma \ref{lem:bumpy}. By the same lemma, $\hat{h}_\tau\in\mathcal{S}^1_{\alpha-1-\varepsilon}$ for any $\varepsilon \in(0,\alpha-2)$. Let $r:\mho_x\rightarrow \mathbb{R}$ be the first return time to $\mho_x$ under the flow $\varphi$.

        As observed above, since $\alpha-1-\varepsilon>1$, the function $\hat{h}_\tau$ has a class in the finite dimensional space $H_{r,\alpha-1-\varepsilon}^1(\Omega_x)\cong \check H^1(\Omega;\mathbb{R})$. Pick an orthonormal basis $\left\{v_1,\dots,v_{d_\mu}\right\}$ for $H_{r,\alpha-1-\varepsilon}^1(\Omega_x)$ which span subspaces $W_i^x:= \mathrm{span}\, v_i$. Let $\mathcal{P}_i:H_{r,\alpha-1-\varepsilon}^1(\Omega_x)\rightarrow W_i^x$ be the associated projections, and define the distribution $\mathcal{D}_i\in H_{\alpha}(X_B)'$ as
        $$\mathcal{D}_i:= h\mapsto \mathrm{sign}(\mathcal{P}_i([\hat{h}_\tau]))\cdot\|\mathcal{P}_i([\hat{h}_\tau])\|,$$
        where the sign is taken relative to the chosen basis, and the class $[\hat{h}_\tau]$ is taken in $H_{r,\alpha-1-\varepsilon}^1(\Omega_x)$.

        Suppose that $h = g\circ\phi- g$. In the product coordinates in $\mathcal{U}_x$, for $z = (t,c)\in\mathcal{U}_x$, define $\Theta(t,c) = g(\varpi^{-1}(c))+ \int_{-\tau_x/5}^tu(y)h(\varpi^{-1}(c))\,dy$. In $\mathcal{U}_x$, $X\Theta = \hat{h}_\tau$. Moreover,
        \begin{equation*}
          \begin{split}
            \Theta\circ \varphi_{\tau_x/5}(z) &= \Theta(\tau_x/5,c) = g(\varpi^{-1}(c)) +  h(\varpi^{-1}(c)) =  g(\varpi^{-1}(c)) + g(\varpi^{-1}(\phi(c))) - g(\varpi^{-1}(c)) \\
            &= \Theta(-\tau_x/5,\varphi\circ\phi\circ\varpi^{-1}(c)) = \Theta\circ \varphi_{r(z)-\tau_x/5}(z),
          \end{split}
        \end{equation*}
        meaning that $\Theta$ can be taken to be leafwise constant between $\varphi_{\tau_x/5}(z)$ and $\varphi_{r(z)-\tau_x/5}(z)$, and so $X\Theta = \hat{h}_\tau = 0$ on $\Omega_x\setminus \mathcal{U}_x$, so $X \Theta = \hat{h}_\tau$ and thus $\mathcal{D}_i(h) = 0$ for all $i$.
        
        Now suppose $\mathcal{D}_i(h) = 0$ for all $i$, that is, $\hat{h}_\tau = X\Theta$ for some $\Theta\in \mathcal{S}^{2}_{\alpha-2-\varepsilon}\subset \mathcal{S}^{1}_{\alpha-2-\varepsilon}$. Then
        \begin{equation*}
          \begin{split}
        h(c) &= h(c)\int_{-\frac{\tau_x}{5}}^{\frac{\tau_x}{5}} u(t)\, dt = h(c)\int_{-\frac{\tau_x}{5}}^{r(\varpi(c))-\frac{\tau_x}{5}} u(t)\, dt = \int_{-\frac{\tau_x}{5}}^{r(\varpi(c))-\frac{\tau_x}{5}}\hat{h}_\tau \circ \varphi_s(\varpi(c))\, dt\\
        & = \Theta\circ \varphi_{r(\varpi(c))-\frac{\tau_x}{5}}(\varpi(c)) -  \Theta\circ \varphi_{-\frac{\tau_x}{5}}(\varpi(c))\\
        & = \Theta\circ \varphi_{-\frac{\tau_x}{5}}(\varpi(\phi(c))) - \Theta\circ \varphi_{-\frac{\tau_x}{5}}(\varpi(c))
          \end{split}
        \end{equation*}
        Define $g:X_B\rightarrow \mathbb{R}$ by $g(c) = \Theta\circ \varphi_{-\frac{\tau_x}{5}}(\varpi(c))$. Then the above calculation shows $h = g\circ \phi - g$. Moreover, by Lemma \ref{lem:pullback},
        $$\|g\|_{\alpha-2-\varepsilon-\varepsilon'}\leq C_{\alpha,\varepsilon,x} \|\Theta\|_{1,\alpha-2-\varepsilon}$$
        for all $ \varepsilon'$ small enough. Moreover, by part (iii) of Proposition \ref{eqn:spaceProps},
        $$C_{\alpha,\varepsilon,x} \|\Theta\|_{1,\alpha-2-\varepsilon}\leq C_{\alpha,\varepsilon,x}' \|X \Theta\|_{1,\alpha-1-\varepsilon} = C_{\alpha,\varepsilon,x}' \|\hat{h}_\tau\|_{1,\alpha-1-\varepsilon} .$$
        Finally, since $1+\alpha-1-\varepsilon < \alpha$, Lemma \ref{lem:bumpy} implies that
        $$C_{\alpha,\varepsilon,x}' \|\hat{h}_\tau\|_{1,\alpha-1-\varepsilon}\leq  C''_{\alpha,x,\varepsilon}\|h\|_\alpha.$$
        Stringing all the inequalities together:
        $$\|g\|_{\alpha-2-\varepsilon}\leq K_{x,\alpha,\varepsilon}\|h\|_\alpha.$$
      \end{proof}
      \section{Cyclic cocycles}
      \label{sec:traces}
        Let $\phi:X_B\rightarrow X_B$ be a uniquely ergodic Vershik transformation on the path space of an ordered Bratteli diagram $(X_B,\leq_r)$. Pick a metric on $X_B$ and set $H_\alpha(X_B)$ to be the space of H\"older functions, i.e., it is the space of continous functions $f:X_B\rightarrow \mathbb{R}$ such that
      $$|f|_\alpha:= \sup_{x\neq y}\frac{|f(x)-f(y)|}{d(x,y)^\alpha}<\infty.$$
      This space is endowed with the norm $\|f\|_\alpha:= \|f\|_\infty + |f|_\alpha$. There are two things to point out about the norm. First, define
      $$\triangle_\phi := \max\left\{ \max_{x\neq y} \frac{d(\phi(x),\phi(y))}{d(x,y)}, \max_{x\neq y} \frac{d(\phi^{-1}(x),\phi^{-1}(y))}{d(x,y)}   \right\}.$$
      Note that $\triangle_\phi = 1$ if and only if $\phi$ is an isometry. Otherwise $\triangle_\phi>1$. In either case, note that for any $k\in\mathbb{Z}$,
      \begin{equation}
        \label{eqn:HoldComp}
        |f\circ \phi^k|_\alpha \leq \triangle^{\alpha|k|}_\phi |f|_\alpha.
      \end{equation}
      Second, note that by writing $f(x)g(x) - f(y)g(y) = f(x)(g(x)-g(y))+ g(y)(f(x)-f(y))$ it immediately follows that $|fg|_\alpha\leq \|f\|_\infty |g|_\alpha + \|g\|_\infty |f|_\alpha$. Using this, along with $\|fg\|_\infty\leq \|f\|_\infty\|g\|_\infty$, leads to
      $$\|fg\|_\alpha\leq \|f\|_\alpha \|g\|_\alpha.$$
      Denote by $u$ be the unitary operator such that $ufu^* = \beta(f):= f\circ \phi^{-1}$. Define
      \begin{equation*}
        \begin{split}
          \ell^1_\alpha(X_B)&:=\left\{f:\mathbb{Z}\rightarrow H_\alpha(X_B)\hspace{.2in}|\hspace{.2in} \|f\|_{\ell^1_\alpha}:= \sum_{k\in\mathbb{Z}}\|f(k)\|_\alpha < \infty \right\}\\
          &\subset \ell^1(X_B) := \ell^1(\mathbb{Z},C(X_B))\subset \mathcal{A}_\phi:=C(X_B)\underset{\phi}{\times} \mathbb{Z}
        \end{split}
      \end{equation*}
      where I am somewhat abusing the notation, identifying $\ell^1_\alpha(X_B)$ with its integrated form
      $$\left\{\sum_{k\in\mathbb{Z}}f(k)u^k: \sum_{k\in\mathbb{Z}}\|f(k)\|_\alpha < \infty \right\}.$$
      Define
      $$W^\infty_\alpha(X_B) = \left\{
      \begin{array}{ll}   \displaystyle  f\in \ell^1_\alpha(X_B) : \sum_{k\in\mathbb{Z}}\|f(k)\|_\alpha \triangle^{(\alpha+s)|k|}_\phi<\infty\mbox{ for all }s>0  &\mbox{ if } \triangle_\phi>1, \\
     \displaystyle    f\in \ell^1_\alpha(X_B) : \sum_{k\in\mathbb{Z}}\|f(k)\|_\alpha (1+|k|)^s<\infty\mbox{ for all }s>0 &\mbox{ if } \triangle_\phi=1, \end{array} \right. $$
      and the families of seminorms
      $$\mu_q^\alpha(f):= \left\{
      \begin{array}{ll}   \displaystyle  \sup_{N\in\mathbb{N}} \left\{ \triangle_\phi^{(\alpha+q)N}\|(1-P_N)f\|_{\ell^1_\alpha(X_B)}\right\} &\mbox{ if } \triangle_\phi>1, \\
     \displaystyle    \sup_{N\in\mathbb{N}} \left\{ N^q\|(1-P_N)f\|_{\ell^1_\alpha(X_B)}\right\} &\mbox{ if } \triangle_\phi=1. \end{array} \right. $$
      where $P_N$ is the projection onto the coordinates $i$ such that $|i|\leq N$, and 
      $$\|f\|_{s,\alpha}:= \left\{
      \begin{array}{ll}   \displaystyle \sum_{k\in\mathbb{Z}}\|f(k)\|_\alpha\triangle_\phi^{(\alpha+s)|k|} &\mbox{ if } \triangle_\phi>1, \\
     \displaystyle  \sum_{k\in\mathbb{Z}}\|f(k)\|_\alpha(1+|k|)^s &\mbox{ if } \triangle_\phi=1. \end{array} \right. $$
      In either case, note that $W^\infty_\alpha(X_B)$ is the set of $f\in\ell^1_\alpha(X_B)$ such that $\|f\|_{s,\alpha}<\infty$ for all $s>0$ and that set $W^\infty_\alpha(X_B)$ is endowed with the product inherited from $\ell^1(X_B)$ defined by
      $$\left(\sum_{k\in\mathbb{Z}}a(k)u^k\right) \left(\sum_{k\in\mathbb{Z}}b(k)u^k\right) = \sum_{k\in\mathbb{Z}}a*b(k)u^k = \sum_{k\in\mathbb{Z}}\sum_{i\in\mathbb{Z}}a(i)b(k-i)\circ \phi^{-i}u^k.$$
      \begin{lemma}
        \label{lem:algebra}
        $f\in W^\infty_\alpha(X_B)$ if and only if $\mu_q^\alpha(f)<\infty$ for all $q\in\mathbb{N}$. Moreover $W^\infty_\alpha(X_B)$ is a unital involutive Fr\'echet $*$-algebra.
      \end{lemma}
      \begin{proof}
        The case when $\triangle_\phi=1$ is essentially already proved in \cite[\S 4]{FLLP:triples}. The case to be proved is when $\triangle_\phi>1$, so assume $\triangle_\phi>1$. Let $f\in W^\infty_\alpha(X_B)$. Then for $N\in\mathbb{N}$:
        $$\triangle_\phi^{(\alpha+q)N}\| (1-P_N)f \|_{\ell^1_\alpha(X_B)} = \triangle_\phi^{(\alpha+q)N}\sum_{|k|>N}\|f(k)\|_\alpha\leq \sum_{|k|>N}\|f(k)\|_\alpha \triangle_\phi^{(\alpha+q)|k|} \leq \|f\|_{q,\alpha}$$
        which shows that $\mu_q^\alpha(f)\leq \|f\|_{q,\alpha}<\infty$. Now suppose that $f$ satisfies $\mu_q^\alpha(f)<\infty$ for all $q>0$. Then
        \begin{equation}
          \begin{split}
            \|f\|_{q,\alpha} &= \sum_{k\in\mathbb{Z}} \|f(k)\|_\alpha\triangle_\phi^{(\alpha+q)|k|}\leq \sum_{k\in\mathbb{Z}}\sum_{\ell\geq |k|} \|f(\ell)\|_\alpha \triangle_\phi^{(\alpha+q)|k|} \\
            &= \sum_{k\in\mathbb{Z}}\triangle_\phi^{(\alpha+q)|k|} \|(1-P_{|k|-1})f\|_{\ell^1_\alpha} = \triangle_\phi^{(\alpha+q)}\sum_{k\in\mathbb{Z}}\triangle_\phi^{(\alpha+q)(|k|-1)} \|(1-P_{|k|-1})f\|_{\ell^1_\alpha} \\
            &= \triangle_\phi^{(\alpha+q)}\sum_{k\in\mathbb{Z}} \triangle_\phi^{-(\alpha+q)(|k|-1)} \triangle_\phi^{(\alpha+\alpha+2q)(|k|-1)} \|(1-P_{|k|-1})f\|_{\ell^1_\alpha} \\
            &= \triangle_\phi^{(\alpha+q)}\sum_{k\in\mathbb{Z}} \triangle_\phi^{-(\alpha+q)(|k|-1)} \mu_{2q+\alpha}^\alpha(f) \\
            &\leq \frac{2\triangle_\phi^{2(\alpha+q)}}{1-\triangle_\phi^{-(\alpha+q)}}\mu_{2q+\alpha}^\alpha(f) < \infty
          \end{split}
        \end{equation}
        for all $q>0$, and so the first part of the Lemma is proved.
        That $W^\infty_\alpha(X_B)$ is invariant under the $*$-involution follows from
        \begin{equation}
          \label{eqn:involution}
            \|a^*\|_{s,\alpha} = \sum_{k\in\mathbb{Z}}\|\overline{a(-k)\circ\phi^{-k}}\|_\alpha \triangle_\phi^{(\alpha+s)|k|}\leq  \sum_{k\in\mathbb{Z}}\|\overline{a(-k)}\|_\alpha \triangle_\phi^{(\alpha+s+\alpha)|k|} = \|a\|_{s+\alpha,\alpha}
        \end{equation}
        for all $s>0$, where (\ref{eqn:HoldComp}) was used, and so $a^*\in W^\infty_\alpha(X_B)$. Finally, for $a,b\in W^\infty_\alpha(X_B)$:
        \begin{equation}
          \label{eqn:product}
          \begin{split}
            \|ab\|_{s,\alpha} &= \sum_{k\in\mathbb{Z}} \|a*b(k)\|_\alpha \triangle_\phi^{(\alpha+s)|k|}  \leq \sum_{k\in\mathbb{Z}} \sum_{j\in\mathbb{Z}}\|a(j) b(k-j)\circ\varphi^{-j}\|_\alpha \triangle_\phi^{(\alpha+s)|k|} \\
            & \leq \sum_{k\in\mathbb{Z}} \sum_{j\in\mathbb{Z}}\|a(j) b(k-j)\circ\varphi^{-j}\|_\alpha \triangle_\phi^{(\alpha+s)|j|} \triangle_\phi^{(\alpha+s)|k-j|}\\
            & \leq \sum_{j\in\mathbb{Z}} \|a(j)\|_\alpha \triangle_\phi^{(\alpha+s+\alpha)|j|}  \sum_{k\in\mathbb{Z}}  \|b(k-j)\|_\alpha  \triangle_\phi^{(\alpha+s)|k-j|}   \\
            &\leq \|a\|_{s+\alpha,\alpha} \|b\|_{s,\alpha} < \infty
          \end{split}
        \end{equation}
        for all $s>0$, and so $ab\in W^\infty_\alpha(X_B)$.
      \end{proof}
      \begin{proposition}
        \label{prop:SI}
        $W^\infty_\alpha(X_B)$ is a dense $*$-subalgebra which is stable under holomorphic functional calculus.
      \end{proposition}
      
      \begin{proof}[Proof of Proposition \ref{prop:SI}]
        The proof will follow a an argument by Jolissaint \cite{J:rapid} which was slightly expanded in \cite[Proposition 4.8]{FLLP:triples}. That argument applies to the case $\triangle_\phi=1$, so here the case $\triangle_\phi>1$ needs to be treated. First, note that for $a,b\in W^{\infty}_\alpha(X_B)$ a calculation similar to (\ref{eqn:product}) gives
        \begin{equation}
          \label{eqn:prod2}
          \|ab\|_{\ell^1_\alpha(X_B)} \leq \|a\|_{0,\alpha}\|b\|_{\ell^1_\alpha(X_B)}.
        \end{equation}
        Now consider $f = 1-h\in W^\infty_\alpha(X_B)$ such that $\|f\|_{\alpha,\alpha}<\frac{1}{2}$. Since $\|f\|_{\ell^1}\leq \|f\|_{\ell^1_\alpha} \leq \|f\|_{\alpha,\alpha}$, the element $h\in W^\infty_\alpha$ is invertible in $\ell^1(X_B)$ with inverse $h^{-1} = \sum_{n\geq 0} f^n$. The first step towards the proof of the proposition is to show that $h^{-1}\in W^\infty_\alpha$.

        So let $P = P_{N^2}:\ell^1(X_B)\rightarrow C_c(\mathbb{Z},C(X_B))$ be the projection onto the coordinates $i$ such that $|i|\leq N^2$ for some $N$. The first step is to note that $f^n$ can be rewritten as
        \begin{equation}
          \label{eqn:jollisaint}
          \begin{split}
            f^n &= (1-P)f^n+Pf^n = (1-P)f^n+Pf\cdot f^{n-1} \\
            &=  (1-P)f^n+Pf ( (1-P)f^{n-1}+Pf^{n-1}) \\
            &= (1-P)f^n+Pf (1-P)f^{n-1}+(Pf)^2\cdot f^{n-2} \\
            &= (1-P)f^n+Pf (1-P)f^{n-1}+(Pf)^2( (1-P)f^{n-2}+Pf^{n-2}) \\
            &= (1-P)f^n+Pf (1-P)f^{n-1}+(Pf)^2 (1-P)f^{n-2}+Pf^3 f^{n-3} \\
            &\;\; \vdots \\
            &= (1-P)f^n+Pf (1-P)f^{n-1}+ \cdots + (Pf)^{n-1}(1-P)f + (Pf)^n,
          \end{split}
        \end{equation}
        in other words,
        $$f^n = \sum_{k=0}^n (P_{N^2}f)^k(1-P_{N^2})f\cdot f^{n-k-1}.$$
        Thus, for $n\leq N^2$, using the the triangle inequality and (\ref{eqn:prod2}), $(1-P_N)f^n$ can be bounded as
        \begin{equation*}
          \begin{split}
            \|(1-P_N)f^n\|_{\ell^1_\alpha} &\leq\sum_{k=0}^n \| (P_{N^2}f)^k(1-P_{N^2})f\cdot f^{n-k-1}\|_{\ell^1_\alpha}\\
            &\leq\sum_{k=0}^n \| (P_{N^2}f)^k\|_{0,\alpha}\,\|(1-P_{N^2})f\|_{0,\alpha}\, \| f^{n-k-1}\|_{\ell^1_\alpha} \\
            &\leq \sum_{k=0}^n \|P_{N^2}f\|^{k}_{\alpha,\alpha}\cdot \|(1-P_{N^2})f\|_{0,\alpha}\cdot \|f\|^{n-k+1}_{0,\alpha}\\
            &\leq \frac{n}{2^n}\|(1-P_{N^2})f\|_{0,\alpha}.
          \end{split}
        \end{equation*}
        Thus
        \begin{equation*}
          \begin{split}
            \left\|(1-P_N)\sum_{n>0}f^n\right\|_{\ell^1_\alpha} &\leq \sum_{n=1}^{N^2} \left\|(1-P_N)f^n  \right\|_{\ell^1_\alpha} +  \sum_{n>N^2} \left\|(1-P_N)f^n  \right\|_{\ell^1_\alpha} \\
            &\leq \sum_{n=1}^{N^2} n2^{-n} \left\|(1-P_{N^2})f  \right\|_{0,\alpha}  +  \sum_{n>N^2} \left\|(1-P_N)f^n  \right\|_{\ell^1_\alpha} \\
            &\leq N^4\|(1-P_{N^2})f\|_{0,\alpha} +  \frac{4\|f\|_{\ell^1_\alpha}}{2^{N^2}}.
          \end{split}
        \end{equation*}
        Thus, multiplying by $\triangle_\phi^{(\alpha+q)N}$:
        \begin{equation*}
          \begin{split}
            \triangle_\phi^{(\alpha+q)N} \left\|(1-P_N)\sum_{n> 0}f^n\right\|_{\ell^1_\alpha}&\leq \triangle_\phi^{(\alpha+q)N}  N^4\|(1-P_{N^2})f\|_{0,\alpha} +  \triangle_\phi^{(\alpha+q)N}\frac{4\|f\|_{\ell^1_\alpha}}{2^{N^2}} \\
            &\leq C_q \triangle_\phi^{(\alpha+2q)N}  \|(1-P_{N})f\|_{0,\alpha} +  \triangle_\phi^{(\alpha+q)N}\frac{4\|f\|_{\ell^1_\alpha}}{2^{N^2}} \\
            &\leq C_q \mu^\alpha_{\alpha+2q}(f) + \sup_N \left\{ \triangle_\phi^{(\alpha+q)N}\frac{4\|f\|_{\ell^1_\alpha}}{2^{N^2}} \right\} < \infty,
          \end{split}
        \end{equation*}
        where $C_q:= \sup_{N>0} \triangle^{-qN}N^4$, since the last term is bounded for any $q$. This implies that
        $$\mu^\alpha_q\left(\sum_{n> 0}f^n \right)<\infty$$
        for all $q>0$, and so $\sum_{n> 0}f^n \in W^\infty_\alpha(X_B)$ by Lemma \ref{lem:algebra}. As such, $\sum_{n\geq 0}f^n$ is also in $W^\infty_\alpha$ and thus $h = 1-f$ is invertible in $W^\infty_\alpha$ whenever $\|1-h\|_{\alpha,\alpha}< \frac{1}{2}$.

        Now suppose that $h\in W^\infty_\alpha(X_B)$ is invertible in $\ell^1(X_B)$. The goal is to show that $h^{-1}$ is in fact an element of $W^\infty_\alpha(X_B)$. Following the end of the proof of \cite[Proposition 4.8]{FLLP:triples}, note that $1_{\ell^1(X_B)} = 1_{W^\infty_\alpha(X_B)} = h\cdot h^{-1}$. Since $W^\infty_\alpha(X_B)$ is dense in $\ell^1(X_B)$ taking $h'\in W^\infty_\alpha(X_B)$ with $\|h^{-1}-h'\|_{\alpha,\alpha}<\frac{1}{2\|h\|_{2\alpha,\alpha}}$, it follows that
        $$\|1-h\cdot h'\|_{\alpha,\alpha} = \|h\cdot h^{-1} - h\cdot h'\|_{\alpha,\alpha}\leq \|h\|_{2\alpha,\alpha}\cdot \|h^{-1}-h'\|_{\alpha,\alpha}<\frac{1}{2}.$$
        Since both $h\cdot h'$ and $1-h\cdot h'$ are both in $W^\infty_\alpha(X_B)$, by the arguments above, $(h\cdot h')^{-1}\in W^\infty_\alpha(X_B)$. Thus, $h^{-1} = h'(h\cdot h')^{-1}$ is an element of $W^\infty_\alpha(X_B)$. It follows that $W^\infty_\alpha(X_B)$ is a \textbf{spectral invariant} subalgebra of $\mathcal{A}_\phi$, that is, it is stable under holomorphic functional calculus (see \cite[\S 4]{FLLP:triples}).
      \end{proof}
      \begin{proof}[Proof of Theorem \ref{thm:main2}]
        Theorem \ref{thm:main} gives for $\mu$-almost every $x = (B,\leq_r)$ a collection of $d_\mu$ $\phi$-invariant distributions $\mathcal{D}_i \in H_\alpha(X_B)'$. Each distribution defines a trace on $W^\infty_\alpha(X_B)$ by
      $$\tau_i: \sum_{k\in\mathbb{Z}}a(k)u^k\mapsto \mathcal{D}_i(a(0)).$$
        By Proposition \ref{prop:SI}, $W^\infty_\alpha(X_B)$ is stable under holomorphic functional calculus, and thus its inclusion into $\mathcal{A}_\phi$ induces an isomorphism on $K_0$ \cite[Appendix 3.C]{Connes:book}. As such, each invariant distribution $\mathcal{D}_i$ defines a trace $\tau_i:K_0(\mathcal{A}_\phi)\rightarrow \mathbb{R}$.
      \end{proof}
      \bibliographystyle{amsalpha}
      \bibliography{biblio}      
\end{document}